\documentclass[11pt]{amsart}
\usepackage[paperheight=11in,paperwidth=8.5in,left=1in,top=1in,right=1in,bottom=1in,heightrounded]{geometry}
\usepackage[pdfstartpage={},pdfstartview={FitH},pdftitle={},pdfauthor={},pdfsubject={},pdfcreator={},pdfproducer={},pdfkeywords={},pagebackref=false,colorlinks=true,linkcolor={red!87.5!black},citecolor={blue!50!black},linktocpage=true]{hyperref}
\hypersetup{breaklinks=true}
\usepackage{accents}
\usepackage{afterpage}
\usepackage{amssymb}
\usepackage{appendix}
\usepackage{array}
\usepackage{bm}
\usepackage{bbm}
\usepackage{braket}
\usepackage{calc}
\usepackage[margin=0cm,font=small,labelfont=bf]{caption}
\usepackage{centernot}
\usepackage[noadjust]{cite}
\usepackage{dsfont}
\usepackage{enumitem}
\usepackage{float}
\restylefloat{figure}
\usepackage{graphicx}
\usepackage{indentfirst}
\usepackage{mathdots}
\usepackage{mathrsfs}
\usepackage{mathtools}
\usepackage{mleftright}
\usepackage{pgf}
\usepackage{scalerel}
\usepackage{stackengine}
\usepackage{tikz}
\usepackage{todonotes}
\usetikzlibrary{arrows,shapes}
\usepackage{tikz-cd}
\allowdisplaybreaks[4]

\newtheorem{thm}{Theorem}[section]

\newtheorem{cor}[thm]{Corollary}

\newtheorem{lemma}[thm]{Lemma}
\newtheorem{prop}[thm]{Proposition}

\theoremstyle{definition}
\newtheorem{defn}[thm]{Definition}
\newtheorem{eg}[thm]{Example}

\theoremstyle{remark}
\newtheorem{rem}[thm]{Remark}

\DeclareMathOperator{\Tr}{Tr}

\DeclareMathOperator{\real}{Re}
\DeclareMathOperator{\imag}{Im}

\newcommand{\source}{\operatorname{src}}
\newcommand{\target}{\operatorname{tar}}

\newcommand{\indep}{\perp \!\!\! \perp}
\newcommand{\N}{\mathbb{N}}

\newcommand{\R}{\mathbb{R}}
\newcommand{\C}{\mathbb{C}}

\newcommand{\sphr}{\mathbb{S}}
\newcommand{\pr}{\mathbb{P}}
\newcommand{\E}{\mathbb{E}}
\newcommand{\pto}{\stackrel{\pr}{\to}}

\newcommand{\aseq}{\stackrel{\op{a.s.}}{=}}

\newcommand{\op}[1]{\operatorname{#1}}
\newcommand{\mcal}[1]{\mathcal{#1}}

\newcommand{\msr}[1]{\mathscr{#1}}
\newcommand{\mbf}[1]{\mathbf{#1}}
\newcommand{\indc}[1]{\mathds{1}\left\{#1\right\}}
\newcommand{\wtilde}[1]{\widetilde{#1}}

\newcommand{\inn}[2]{\langle #1, #2 \rangle}
\newcommand{\deq}{\stackrel{d}{=}}

\newcommand{\norm}[1]{\lVert#1\rVert}
\newcommand{\snorm}[1]{\left\lVert#1\right\rVert}

\DeclareMathOperator{\matn}{Mat}

\DeclareMathOperator{\Wig}{Wigner}
\DeclareMathOperator{\rbm}{RBM}

\DeclareMathOperator*{\argmin}{arg\,min}
\DeclareRobustCommand{\SkipTocEntry}[5]{} 

\makeatletter
\newcommand\thefontsize[1]{{#1 The current font size is: \f@size pt\par}}
\makeatother

\begin{document}
\author[Benson Au]{Benson Au}
\thanks{Department of Statistics, University of California, Berkeley, \href{mailto:bensonau@berkeley.edu}{bensonau@berkeley.edu}}
\address{Department of Statistics\\
         University of California, Berkeley\\
         367 Evans Hall \# 3860\\
         Berkeley, CA 94720-3860\\
         USA}
       \email{\href{mailto:bensonau@berkeley.edu}{bensonau@berkeley.edu}}
       
\title[Deformed random band matrices]{BBP phenomena for deformed random band matrices}
\date{\today}

\begin{abstract}\label{sec:abstract}
We study additive finite-rank perturbations of random periodic band matrices under the assumption that the nontrivial eigenvalues of the perturbation do not depend on the dimension. We establish the eigenvalue/eigenvector BBP transition in this model for band widths $b_N \gg N^\varepsilon$. Our analysis relies on moment method calculations for general vector states.
\end{abstract}

\maketitle
\tableofcontents

\section{Introduction}\label{sec:intro}

Understanding the spectral statistics of random matrices is a fundamental problem at the interface of mathematics, physics, and statistics. This confluence can already be observed in the classical Wigner ensemble, a mean-field model originally proposed by Wigner as a tractable proxy for the Hamiltonian of a large quantum system. In the intervening years, the definition of a Wigner matrix has become increasingly general. For concreteness, we state our working definition below.

\begin{defn}[Wigner matrix]\label{defn:wigner_matrix}
Let $(\mbf{X}_N(i, j) : 1 \leq i \leq j \leq N \in \N)$ be a family of independent random variables such that
\begin{enumerate}[label=(\roman*)]
\item \label{cond:off_diagonal} the off-diagonal entries $(i < j)$ are complex-valued, centered, and of variance $\sigma^2$;
\item \label{cond:diagonal} the diagonal entries $(i = j)$ are real-valued and of finite variance;
\item \label{cond:moments} we have a strong uniform control on the moments: for any $m \in \N$,
  \begin{equation}\label{eq:finite_moments}
    \sup_{1 \leq i \leq j \leq N \in \N} \E[|\mbf{X}_N(i, j)|^m] < \infty.
  \end{equation}
\end{enumerate}
We call the random Hermitian matrix defined by $\mbf{W}_N(i, j) = \frac{1}{\sqrt{N}} \mbf{X}_N(i, j)$ a \emph{normalized Wigner matrix of variance $\sigma^2$} and use the notation $\mbf{W}_N \deq \Wig(N, \sigma^2)$. When the context is clear, we simply refer to a Wigner matrix. Hereafter, when we refer to a Wigner matrix $\mbf{W}_N$, we implicitly refer to a sequence of Wigner matrices $(\mbf{W}_N)_{N \in \N}$.
\end{defn}

Being Hermitian, we can order the eigenvalues of a Wigner matrix $\lambda_1(\mbf{W}_N) \leq \cdots \leq \lambda_N(\mbf{W}_N)$. The natural question of the limiting distribution of these eigenvalues was settled by Wigner under some simplifying assumptions on the distributions of the entries \cite{Wig55,Wig58} and by Pastur in the general case with the moment assumption \ref{cond:moments} replaced by the much weaker Lindeberg condition \cite{Pas72} (see also \cite[Theorem 2.9]{BS10}): if $\mbf{W}_N \deq \Wig(N, \sigma^2)$, then the empirical spectral distribution $\mu_{\mbf{W}_N} = \frac{1}{N} \sum_{k \in [N]} \delta_{\lambda_k(\mbf{W}_N)}$ converges weakly almost surely to the semicircle distribution $\mu_{\mcal{SC},\sigma^2}(dx) = \frac{\indc{|x| \leq 2\sigma}}{2\pi\sigma^2}\sqrt{4\sigma^2-x^2} \, dx$.

The semicircle law governs the \emph{global} behavior of the eigenvalues; however, the physical interpretation primarily concerns the \emph{local} eigenvalue statistics of the matrix, in particular their conjectured universality \cite{Wig67}. Eigenvector statistics, as pioneered by Anderson \cite{And58}, are a related line of inquiry. In particular, the Anderson tight binding model exhibits localized eigenfunctions \cite{FS83,FMSS85,AM93,Aiz94} and Poisson local eigenvalue statistics \cite{Min96}. Following a long line of work, the universality phenomenon for Wigner matrices is now well-understood \cite{BGK17,EY17}: some highlights include GOE/GUE universality for the local eigenvalue statistics and complete delocalization of the eigenvectors. Random band matrices emerge as a natural interpolative model to study the transition between these two phases \cite{Bou18}.

\begin{defn}[Random band matrix]\label{defn:rbm}
Let $(\mbf{X}_N(i, j) : 1 \leq i \leq j \leq N \in \N)$ be as in Definition \ref{defn:wigner_matrix}. For band widths $(b_N : N \in \N) \subset \N_0$, we define
  \begin{equation*}
    \xi_N = \min\{2b_N + 1, N\}.
  \end{equation*}
Similarly, we define the $N$-periodic distance
  \begin{equation*}
    |i-j|_N = \min\{|i-j|, N - |i-j|\}.
  \end{equation*}
A \emph{periodic $(0,1)$-band matrix of band width $b_N$} is a real symmetric matrix $\mbf{B}_N$ with entries
\begin{equation}\label{eq:band_matrix}
        \mbf{B}_N(i, j) = \indc{|i - j|_N \leq b_N}.
\end{equation}
We call the random Hermitian matrix defined by
  \begin{equation}\label{eq:rbm}
    \mbf{\Xi}_N = \frac{1}{\sqrt{\xi_N}} \mbf{B}_N \circ \mbf{X}_N
  \end{equation}
a \emph{normalized periodic random band matrix of variance $\sigma^2$ and band width $b_N$} and use the notation $\mbf{\Xi}_N \deq \rbm(N, \sigma^2, b_N)$. Here, $\circ$ denotes the entrywise product. When the context is clear, we simply refer to a random band matrix. Hereafter, when we refer to a random band matrix $\mbf{\Xi}_N$, we implicitly refer to a sequence of random band matrices $(\mbf{\Xi}_N)_{N \in \N}$.
\end{defn}

A long-standing conjecture proposes a dichotomy for random band matrices: delocalization and Wigner local statistics for large band widths; localization and Poisson local statistics for small band widths; and a sharp transition around the critical band width rate $b_N \asymp \sqrt{N}$ \cite{CMI90,FM91}. Recent progress has established delocalization (in fact, quantum unique ergodicity) for $b_N \gg N^{3/4}$ \cite{BYY20} and localization for $b_N \ll N^{1/4}$ \cite{CPSS22,CS22}.

Wigner matrices also appear in statistics, where deformed versions are studied as a prototype of a spiked model. Here, spectral properties can be used to differentiate the spiked model from the null case \cite{BBP05,Pec06}, a phenomenon known as the BBP transition. To explain this transition, we first review the relevant results in the null case of a Wigner matrix. We assume that the off-diagonal entries in \ref{cond:off_diagonal} are i.i.d.\@ and similarly for the diagonal entries in \ref{cond:diagonal}, but we no longer assume the existence of moments as in \ref{cond:moments}. Recall that
\begin{enumerate}[leftmargin=*, label = (W\arabic*)]
    \item \label{wigner_edge} The extremal eigenvalues are known to converge to the edge of the support of the semicircle distribution iff the off-diagonal entries have a finite fourth moment $\E[|\mbf{X}_N(1, 2)|^4] < \infty$ \cite{BY88,BS10}. So, for example, $\lim_{N \to \infty} \lambda_1(\mbf{W}_N) \aseq -2 \sigma$; however, without a finite fourth moment, $\liminf_{N \to \infty} \lambda_1(\mbf{W}_N) \aseq -\infty$.
    \item \label{wigner_fluctuations} The fluctuations of the extremal eigenvalues in the GOE/GUE were found in \cite{TW94,TW96} and shown to be universal in \cite{Sos99} assuming sub-Gaussianity of the entries. The optimal rate of decay for universality was found in \cite{LY14} to be $\lim_{s \to \infty} s^4\pr(|\mbf{X}_N(1, 2)| > s) = 0$. For example, $\lim_{N \to \infty} \pr(N^{2/3}(\lambda_N(\mbf{W}_N) - 2\sigma) \leq s\sigma) = F_\beta(s)$, where $F_\beta$ is the CDF of the Tracy-Widom distribution of parameter $\beta$. In particular, we note the $N^{-2/3}$ scale of the fluctuations.
    \item \label{wigner_eigenvectors} As noted before, the ($\ell^2$-normalized) eigenvectors $(\mbf{w}_N^{(k)})_{k \in [N]}$ of $\mbf{W}_N$ are completely delocalized. For example, if we assume finite moments as in \ref{cond:moments}, then for any $\varepsilon, D > 0$, $\pr(\max_{k \in [N]} \norm{\mbf{w}_N^{(k)}}_\infty \geq N^{\varepsilon - 1/2}) \leq N^{-D}$ \cite[Theorem 1.2.10]{BGK17}.   
\end{enumerate}

The spiked Wigner model introduces an additive perturbation $\mbf{A}_N$ to our matrix. We assume that $\mbf{A}_N$ is self-adjoint and of fixed rank $r_N \equiv r$. We further assume that the nontrivial eigenvalues of $\mbf{A}_N$ do not depend on $N$: we denote them by $\theta_1 \leq \cdots \leq \theta_r$. Since the perturbation is finite-rank, the empirical spectral distribution of the spiked model $\mbf{M}_N = \mbf{W}_N + \mbf{A}_N$ still converges to the semicircle distribution. The presence of $\mbf{A}_N$ can however be detected by the extremal spectral statistics (cf. \ref{wigner_edge}-\ref{wigner_eigenvectors}). Recall that
\begin{enumerate}[leftmargin=*, label=(S\arabic*)]
    \item \label{spiked_edge} Each eigenvalue $\theta_s$ of $\mbf{A}_N$ such that $|\theta_s| > \sigma$ creates an outlying eigenvalue in $\mbf{M}_N$. In particular, if $L_{-\sigma} = \#(\{s \in [r]: \theta_s < -\sigma\})$ and $L_{+\sigma} = \#(\{s \in [r]: \theta_s > \sigma\})$, then
    \begin{align*}
        \lim_{N \to \infty} \lambda_k(\mbf{M}_N) &\aseq \theta_k + \frac{\sigma^2}{\theta_k} < -2\sigma, \qquad \forall k \in [L_{-\sigma}]; \\
        \lim_{N \to \infty} \lambda_{L_{-\sigma} + 1}(\mbf{M}_N) &\aseq -2\sigma; \\
        \lim_{N \to \infty} \lambda_{N+1-k}(\mbf{M}_N) &\aseq \theta_{r+1-k} + \frac{\sigma^2}{\theta_{r+1-k}} > 2\sigma, \qquad \forall k \in [L_{+\sigma}]; \\
        \lim_{N \to \infty} \lambda_{N - L_{+\sigma}}(\mbf{M}_N) &\aseq 2\sigma.
    \end{align*}
    \item \label{spiked_fluctuations} The fluctuations of the outlying eigenvalues are nonuniversal. We omit the precise statement of the result in this case and simply note the $N^{-1/2}$ scale of the fluctuations.
    \item \label{spiked_eigenvectors_I} The eigenspace of an outlying eigenvalue in $\mbf{M}_N$ has nontrivial alignment with the eigenspace of the corresponding eigenvalue in $\mbf{A}_N$. In particular, let $\mbf{m}_N^{(k)}$ be a unit eigenvector associated with the eigenvalue $\lambda_k(\mbf{M}_N)$. If $k \in [L_{-\sigma}]$, then 
    \[
      \lim_{N \to \infty} \snorm{P_{\ker(\theta_k \mbf{I}_N - \mbf{A}_N)}(\mbf{m}_N^{(k)})}_2^2 \aseq 1 - \frac{\sigma^2}{\theta_k^2},
    \]
    where $P_{\ker(\theta_k \mbf{I}_N - \mbf{A}_N)}$ denotes the orthogonal projection onto the eigenspace $\ker(\theta_k \mbf{I}_N - \mbf{A}_N)$; however, if $\theta_{k'} \neq \theta_k$, then
    \[
      \lim_{N \to \infty} \snorm{P_{\ker(\theta_{k'} \mbf{I}_N - \mbf{A}_N)}(\mbf{m}_N^{(k)})}_2 \aseq 0.
    \]
    Similarly, if $k \in [L_{+\sigma}]$, then one replaces all instances of $k$ in the superscripts with $N+1-k$ and all instances of $k$ in the subscripts with $r+1-k$ in the above.
    \item \label{spiked_eigenvectors_II} On the other hand, if $\theta_k$ does not meet the threshold in \ref{spiked_edge} for the creation of an outlier, then the eigenspace of the associated eigenvalue in $\mbf{M}_N$ is asymptotically orthogonal to $\ker(\mbf{A}_N)^\perp$. In particular, if $\theta_k \in [-\sigma, 0)$, then
    \[
      \lim_{N \to \infty} \snorm{P_{\ker(\theta_{k'} \mbf{I}_N - \mbf{A}_N)}(\mbf{m}_N^{(k)})}_2 \aseq 0, \qquad \forall k' \in [r].
    \]
    Similarly, if $\theta_{r + 1 - k} \in (0, \sigma]$, then one replaces $\mbf{m}_N^{(k)}$ with $\mbf{m}_{N}^{(N+1-k)}$ in the above.
\end{enumerate}

The outlier phenomenon in \ref{spiked_edge} was first proven for the GUE in \cite{Pec06}, extended to Wigner matrices satisfying a Poincar\'{e} inequality in \cite{CDMF09}, and then relaxed to a fourth moment Lindeberg-type condition in \cite{PRS13,RS13} at the cost of convergence in probability. In the complex case, we note that the assumption $\real(\mbf{X}_N(i,j)) \indep \imag(\mbf{X}_N(i,j))$ are identically distributed is present throughout. The works \cite{CDMF09,CDMF12,PRS13,RS13} address the fluctuations touched on in \ref{spiked_fluctuations}. Much finer results are known if one assumes uniform subexponential decay of the entries, in which case one can leverage the isotropic local semicircle law \cite{KY13,KY14}, but we will not discuss this further. The eigenvector alignment in \ref{spiked_eigenvectors_I} was first proven for general unitarily/orthogonally invariant random matrices in \cite{BGN11} and extended to Wigner matrices satisfying a Poincar\'{e} inequality in \cite{Cap13}. The nonalignment in \ref{spiked_eigenvectors_II} was proven for the same invariant ensembles in \cite{BGN11} under the assumption of a rank one perturbation $r = 1$.

In this paper, we study the spiked RBM model $\mbf{M}_N = \mbf{\Xi}_N + \mbf{A}_N$. Our main result proves that the eigenvalue/eigenvector BBP transition persists for band widths $b_N \gg N^{\varepsilon}$.

\begin{thm}\label{thm:spiked_rbm}
Let $\mbf{\Xi}_N$ be a RBM as in Definition \ref{defn:rbm}. If $b_N \gg N^\varepsilon$ for some $\varepsilon > 0$, then the spiked RBM model $\mbf{M}_N = \mbf{\Xi}_N + \mbf{A}_N$ exhibits the eigenvalue/eigenvector BBP transition in \ref{spiked_edge}, \ref{spiked_eigenvectors_I}, and \ref{spiked_eigenvectors_II}.
\end{thm}

We briefly outline the proof of Theorem \ref{thm:spiked_rbm}. In the case of a rank one perturbation $\theta \mbf{a}_N \mbf{a}_N^*$ of a Wigner matrix $\mbf{W}_N$, Noiry computed the limiting spectral measure $\mu_{\theta}$ of $\mbf{W}_N + \theta \mbf{a}_N \mbf{a}_N^*$ with respect to the vector state $\tau_N(\cdot) = \langle \cdot \mbf{a}_N, \mbf{a}_N\rangle$ \cite[Proposition 2]{Noi21}. In particular,
\begin{equation}\label{eq:spectral_measure_with_vector_v}
    \mu_{\theta}(dx) = \frac{\indc{|x| \leq 2\sigma}}{2\pi}\frac{\sqrt{4\sigma^2 - x^2}}{\theta^2 + \sigma^2 - \theta x} \, dx + \indc{|\theta| > \sigma} \left(1-\frac{\sigma^2}{\theta^2}\right)\delta_{\theta + \frac{\sigma^2}{\theta}}(dx).
\end{equation}
The strong convergence \ref{wigner_edge} of $\mbf{W}_N$ and Weyl's interlacing inequality \cite[Theorem 4.3.1]{HJ13} then imply \ref{spiked_edge} and \ref{spiked_eigenvectors_I} in the spiked Wigner model for $r = 1$ \cite[Corollary 3]{Noi21} (in fact, \ref{spiked_eigenvectors_II} also follows from the same calculation). Noiry's proof of \eqref{eq:spectral_measure_with_vector_v} uses the local law in \cite[Theorem 12.2]{KY17}, but he mentions that simpler arguments suffice in the case of standard basis vectors (for example, the resolvent estimates in \cite[Proposition 6.2]{Cap13}).

In contrast to the usual approach to outliers via the resolvent, our analysis relies on moment method calculations for general vector states. In particular, we prove a seemingly innocuous isotropic \emph{global} law in Proposition \ref{prop:isotropic_global_law}. In its simplest form, it states that if $b_N \gg N^\varepsilon$ for some $\varepsilon > 0$, then
\begin{equation}\label{eq:isotropic_global_law_intro}
    \lim_{N \to \infty} \Tr\left(\prod_{s = 1}^r \mbf{x}_{N}^{(s-1)} {\mbf{y}_{N}^{(s)}}^* \mbf{\Xi}_N^{m_s}\right) \aseq \prod_{s = 1}^r \left[\lim_{N \to \infty} \langle \mbf{x}_N^{(s)}, \mbf{y}_N^{(s)} \rangle \lim_{N \to \infty} \frac{1}{N}\E\left[\Tr(\mbf{\Xi}_N^{m_s})\right]\right],
\end{equation}
where $\mbf{x}_N^{(r)} = \mbf{x}_N^{(0)}$. The proof of Theorem \ref{thm:spiked_rbm} now follows. Indeed, let $\mbf{A}_N = \sum_{s = 1}^r \theta_s \mbf{a}_N ^{(s)} {\mbf{a}_N^{(s)}}^*$ be the spectral decomposition of the perturbation. The formula in \eqref{eq:isotropic_global_law_intro} tells us that the limiting moments of $\mbf{\Xi}_N + \mbf{A}_N$ with respect to the vector state $\tau_{N}^{(s')}(\cdot) = \langle \cdot \mbf{a}_N^{(s')}, \mbf{a}_N^{(s')}\rangle$ coincide with the limiting moments of $\mbf{W}_N + \theta_{s'} \mbf{a}_N^{(s')}{\mbf{a}_N^{(s')}}^*$ with respect to the same vector state for any $s' \in [r]$. Since the measure in \eqref{eq:spectral_measure_with_vector_v} is uniquely determined by its moments (being compactly supported), this shows that the limiting spectral measure of $\mbf{\Xi}_N + \mbf{A}_N$ with respect to the vector state $\tau_{N}^{(s')}$ is again given by $\mu_{\theta_{s'}}$. To complete the proof, we use the strong convergence of $\mbf{\Xi}_N$ \cite[Corollary 2.18]{BvH22} and an inductive application of Weyl's interlacing inequality.

\begin{rem}\label{rem:optimal_band_width_rate}
Naturally, one can ask for the optimal band width rate in Theorem \ref{thm:spiked_rbm}. The strategy above proves the eigenvalue/eigenvector BBP transition for $\mbf{\Xi}_N$ whenever we have the isotropic global law and the strong convergence of $\mbf{\Xi}_N$. If $b_N \gg 1$, then the convergence in \eqref{eq:isotropic_global_law_intro} still holds in probability. The limiting factor is then the strong convergence of $\mbf{\Xi}_N$. Here, there is a ``tradeoff between sparsity and integrability of the entries" \cite[Remark 7.13]{BvH22}. For example, if one assumes the uniform bound $\norm{\mbf{X}_N(i, j)}_p \leq (Cp)^{\alpha}$ for some constants $C, \alpha \in [0, \infty)$ independent of $p$, then $\norm{\mbf{\Xi}_N} \pto 2\sigma$ for band widths $b_N \gg [\log(N)]^{6(1+\alpha)}$ \cite[Theorem 1.4]{BGP14}, where $\pto$ denotes convergence in probability (see also \cite[Corollary 2.18]{BvH22}). In the case of Rademacher entries, the convergence $\norm{\mbf{\Xi}_N} \pto 2\sigma$ is known for band widths $b_N \gg \log(N)$ \cite[Theorem 1.4]{Sod10}; for Gaussian entries, the rate $b_N \gg \log(N)$ is in fact optimal \cite[Corollary 4.4]{BvH16}. Thus, one also has \ref{spiked_edge}, \ref{spiked_eigenvectors_I}, and \ref{spiked_eigenvectors_II} \emph{in probability} for spiked Gaussian RBMs for the optimal band width rate $b_N \gg \log(N)$.
\end{rem}

\begin{rem}\label{rem:noncommutative_probability}
A key input to our analysis is the calculation of the limiting spectral measure in \eqref{eq:spectral_measure_with_vector_v}. Given the moment determinacy of $\mu_\theta$, this calculation should in principle be possible purely on the basis of \eqref{eq:isotropic_global_law_intro}. In fact, the relationship between the matrices $(\mbf{x}_N^{(s-1)} {\mbf{y}_N^{(s)}}^*)_{s=1}^r$ and $\mbf{\Xi}_N$ in \eqref{eq:isotropic_global_law_intro} is a particular instance of infinitesimal freeness, a concept introduced in \cite{BS12}. Shlyakhtenko used the infinitesimal framework to give an interpretation for the eigenvalue BBP transition in unitarily invariant ensembles at the level of the $\frac{1}{N}$ correction \cite{Shl18}. This was further developed by Collins, Hasebe, and Sakuma in \cite{CHS18} using their framework of cyclic monotone independence. The recent breakthrough of C\'{e}bron, Dahlqvist, and Gabriel in \cite{CDG22} unifies these and other notions (e.g., conditional freeness \cite{BLS96} and monotone independence \cite{Mur01}) and provides a rigorous derivation of the BBP transition from noncommutative probabilistic methods. In particular, the calculation of \eqref{eq:spectral_measure_with_vector_v} in the case of a rank one perturbation of the GUE can be realized as the monotone convolution $\mu_\theta = \delta_\theta \rhd \mu_{\mcal{SC}, \sigma^2}$ \cite[Section 1.4]{CDG22}. As we will not use this framework, we do not discuss this further.
\end{rem}

We highlight an interesting feature of Theorem \ref{thm:spiked_rbm}. The original proofs of \ref{spiked_edge}, \ref{spiked_eigenvectors_I}, and \ref{spiked_eigenvectors_II} for unitarily/orthogonally invariant random matrices in \cite{BGN11} crucially uses the fact that the eigenvectors of an invariant ensemble are Haar distributed. The authors remark that the proofs could possibly be adapted to random matrices with Haar-like eigenvectors \cite[Remark 2.15]{BGN11}. On the other hand, Theorem \ref{thm:spiked_rbm} still holds in the established localized regime $b_N \ll N^{1/4}$. For example, consider a rank one perturbation $\mbf{A}_N = \theta \mbf{a}_N \mbf{a}_N^*$ with $\theta < -\sigma$. The alignment in \ref{spiked_eigenvectors_I} amounts to the convergence
\[
    \lim_{N \to \infty} |\langle \mbf{a}_N, \mbf{m}_N^{(1)} \rangle|^2 \aseq 1 - \frac{\sigma^2}{\theta^2}.
\]
For $\frac{\sigma^2}{\theta^2}$ small, this implies that the eigenvector $\mbf{m}_N^{(1)}$ of the spiked model $\mbf{M}_N = \mbf{\Xi}_N + \theta \mbf{a}_N \mbf{a}_N^*$ takes on the shape of the eigenvector $\mbf{a}_N$ of the perturbation: if $\mbf{a}_N$ is localized, then so too is $\mbf{m}_N^{(1)}$; if $\mbf{a}_N$ is delocalized, then so too is $\mbf{m}_N^{(1)}$.

We do not address the fluctuations of the outlying eigenvalues in the spiked RBM model in this article (in particular, the analogue of \ref{spiked_fluctuations}). Here, the \emph{scale} of the fluctuations depends on the shape of the perturbing eigenvectors. This will be the subject of future work. See \cite[Section 2]{Au21} for heuristics and simulations.

\begin{rem}\label{rem:k_regular_wigner}
We have stated our results for RBMs, but the isotropic global law holds more generally for \emph{$k_N$-sparse Wigner matrices}. Here, one replaces the $(0, 1)$-band matrix $\mbf{B}_N$ (resp., the normalization term $\sqrt{\xi_N}$) in the entrywise product \eqref{eq:rbm} with the adjacency matrix $\wtilde{\mbf{B}}_N$ of a $k_N$-regular graph on the vertex set $[N]$ (resp., the normalization term $\sqrt{k_N}$). If $k_N \gg N^\varepsilon$ for some $\varepsilon > 0$, then we again have the strong convergence of this model to the semicircle law \cite[Corollary 2.18]{BvH22}. Thus, Theorem \ref{thm:spiked_rbm} extends to $k_N$-sparse Wigner matrices: the minor modifications necessary for the proof are contained in Remarks \ref{rem:expectation_k_regular_wigner} and  \ref{rem:central_moments_k_regular_wigner}.
\end{rem}

\subsubsection*{Acknowledgments}
The author thanks Guillaume C\'{e}bron for bringing his attention to the works \cite{Noi21,CDG22}.
The author also thanks Jorge Garza-Vargas and Shirshendu Ganguly for many helpful conversations. 

\section{Background}\label{sec:background}

Let $\matn_N(\C)$ denote the set of complex $N \times N$ matrices. For a Hermitian matrix $\mbf{H}_N \in \matn_N(\C)$, we write $\mbf{H}_N = \sum_{k = 1}^N \lambda_k(\mbf{H}_N) \mbf{h}_N^{(k)} {\mbf{h}_N^{(k)}}^*$ for its spectral decomposition.

\begin{defn}[Spectral measure with respect to a state $\psi$]\label{defn:spectral_measure}
Let $\psi: \matn_N(\C) \to \C$ be a state (i.e., a positive linear functional such that $\psi(\mbf{I}_N) = 1$). We define the \emph{spectral measure of $\mbf{H}_N$ with respect to $\psi$} as the unique probability measure $\mu_{\mbf{H}_N}^\psi$ such that
\begin{equation*}
    \int x^m \, \mu_{\mbf{H}_N}^{\psi}(dx) = \psi(\mbf{H}_N^m), \qquad \forall m \in \N.
\end{equation*}
\end{defn}
\begin{eg}[Spectral measure with respect to a vector state]\label{eg:spectral_measure_vector}
The empirical spectral distribution $\mu_{\mbf{H}_N} = \frac{1}{N} \sum_{k=1}^N \delta_{\lambda_k(\mbf{H}_N)}$ is the spectral measure with respect to the normalized trace $\mu_{\mbf{H}_N} = \mu_{\mbf{H}_N}^{\frac{1}{N}\Tr}$.
In the case of a vector state $\psi(\cdot) = \langle \cdot \mbf{u}_N, \mbf{u}_N \rangle$, we use the notation $\mu_{\mbf{H}_N}^{\mbf{u}_N}$. The spectral decomposition implies that 
\[
    \mu_{\mbf{H}_N}^{\mbf{u}_N} = \sum_{k = 1}^N |\langle \mbf{u}_N, \mbf{h}_N^{(k)}\rangle|^2 \delta_{\lambda_k(\mbf{H}_N)}.
\]
\end{eg}

More generally, we will need to compute quantities of the form
\[
    \prod_{s = 1}^r \left\langle p_s(\mcal{H}_N) \mbf{x}_N^{(s)}, \mbf{y}_N^{(s)} \right\rangle = \prod_{s = 1}^r \Tr\left(p_s(\mcal{H}_N) \mbf{x}_N^{(s)}{\mbf{y}_N^{(s)}}^*\right),
\]
where $p_s(\vec{z}) \in \C\langle z_i : i \in I \rangle$ is a noncommutative polynomial evaluated on a family of Hermitian matrices $\mcal{H}_N = (\mbf{H}_N^{(i)})_{i \in I}$. We write the inner product as a trace to suggest the usual graphical approach to such calculations, with a slight modification to distinguish the matrix $\mbf{x}_N^{(s)}{\mbf{y}_N^{(s)}}^*$.

\begin{defn}[Graphs of matrices]\label{defn:graphs_of_matrices}
A \emph{multidigraph} $G = (V, E, \source, \target)$ consists of a nonempty set of vertices $V$, a set of edges $E$, and directions $\source, \target: E \to V$ indicating the source and target of each edge. A \emph{test graph} $T = (G, \gamma)$ is a finite multidigraph $G$ with edge labels $\gamma: E \to I$. For a partition $\pi \in \mcal{P}(V)$, we construct the \emph{quotient test graph} $T^\pi = (G^\pi, \gamma^\pi)$ by merging the vertices of $G$ so that $V^\pi = \pi$. The underlying multidigraph $G^\pi = (V^\pi, E^\pi, \source^\pi, \target^\pi)$ and the associated edge labels $\gamma^\pi: E^\pi \to I$ can then be written as 
\begin{enumerate}[label=(\roman*)]
    \item $V^\pi = V/\sim_\pi = \{[v]_{\pi}: v \in V\}$ and $E^\pi = E$;
    \item $\source^\pi(e) = [\source(e)]_{\pi}$ and $\target^\pi(e) = [\target(e)]_{\pi}$; 
    \item $\gamma^\pi = \gamma$.
\end{enumerate}
For convenience, we simply write $G^\pi = (V^\pi, E)$. By a slight abuse of notation, we often speak of a test graph $T$ and its underlying multidigraph $G$ interchangeably. For example, we also use the notation $\msr{V}(T)$ and $\msr{E}(T)$ for the vertex set and the edge set of a test graph respectively.

We can evaluate a test graph $T$ on a family of matrices $\mcal{M}_N = (\mbf{M}_N^{(i)})_{i \in I}$ using the formula
\begin{equation*}
    \chi(T, \mcal{M}_N) := \sum_{\phi: V \to [N]} \prod_{e \in E} \mbf{M}_N^{(\gamma(e))}(\phi(e)), 
\end{equation*}
where $(\phi(e)) := (\phi(\target(e)), \phi(\source(e))) \in [N]^2$. Similarly, we define
\begin{equation*}
    \chi^0(T, \mcal{M}_N) := \sum_{\phi: V \hookrightarrow [N]} \prod_{e \in E} \mbf{M}_N^{(\gamma(e))}(\phi(e)), 
\end{equation*}
where $\phi: V \hookrightarrow [N]$ denotes an injective map. The functions $\chi$ and $\chi^0$ are related by the M\"{o}bius formula
\[
    \chi(T, \mcal{M}_N) = \sum_{\pi \in \mcal{P}(V)} \chi^0(T^\pi, \mcal{M}_N).
\]
\end{defn}

\begin{eg}[Moments]\label{eg:moments}
For a monomial $p(\vec{z}) = z_{i(1)} \cdots z_{i(d)} \in \C\langle z_i : i \in I \rangle$ of degree $d$,
\begin{equation}\label{eq:trace}
  \Tr(p(\mcal{M}_N)) = \chi(C_p, \mcal{M}_N) = \sum_{\pi \in \mcal{P}(\msr{V}(C_p))} \chi^0(C_p^\pi, \mcal{M}_N),
\end{equation}
where $C_p$ is the test graph
\begin{equation}\label{eq:moment_graph}
  \begin{tikzpicture}[shorten > = 2.5pt, baseline=(current  bounding  box.center)]
    \node at (0, 0) {$C_p = $};
    \draw[fill=black] (3,0) circle (1.75pt);
    \draw[fill=black] (1,0) circle (1.75pt);
    \draw[fill=black] (2.5,.866) circle (1.75pt);
    \draw[fill=black] (2.5,-.866) circle (1.75pt);
    \draw[fill=black] (1.5,.866) circle (1.75pt);
    \draw[fill=black] (1.5,-.866) circle (1.75pt);
    \draw[semithick, ->] (3,0) to node[pos=.625, right] {\footnotesize$i(d-1)$\normalsize} (2.5,-.866);
    \draw[semithick, ->] (2.5,-.866) to node[pos=.425,below] {\footnotesize$\cdots$\normalsize} (1.5,-.866);
    \draw[semithick, ->] (1.5,-.866) to node[pos=.375, left] {\footnotesize$i(3)$\normalsize} (1, 0);
    \draw[semithick, ->] (1, 0) to node[pos=.625, left] {\footnotesize$i(2)$\normalsize} (1.5, .866);
    \draw[semithick, ->] (1.5, .866) to node[midway, above] {\footnotesize$i(1)$\normalsize} (2.5, .866);
    \draw[semithick, ->] (2.5, .866) to node[pos=.375, right] {\footnotesize$i(d)$\normalsize} (3, 0);
  \end{tikzpicture}
\end{equation}
\end{eg}

\begin{eg}[Random band matrices]\label{eg:rbm}
Let $\mcal{W}_N = (\mbf{W}_N^{(i)})_{i \in I}$ be a family of independent Wigner matrices $\mbf{W}_N^{(i)} \deq \Wig(N, \sigma_i^2)$ as in Definition \ref{defn:wigner_matrix}. Mixed moments in the family $\mcal{W}_N$ are governed by free independence in the large dimension limit \cite{Voi91,Dyk93}: if $p(\vec{z})$ is a monomial as in Example \ref{eg:moments}, then
\begin{equation}\label{eq:nc_wick_formula}
    \tau(p) := \lim_{N \to \infty} \E\left[\frac{1}{N}\Tr(p(\mcal{W}_N))\right] = \sum_{\pi \in \mcal{NC}_2(d)} \prod_{\{j, k\} \in \pi} \sigma_{i(j)}\sigma_{i(k)} \indc{i(j) = i(k)},
\end{equation}
where $\mcal{NC}_2(d)$ is the set of noncrossing pair partitions of $[d]$. We showed that the same convergence holds for a family $\mcal{Z}_N = (\mbf{\Xi}_N^{(i)})_{i \in I}$ of independent RBMs $\mbf{\Xi}_N^{(i)} \deq \rbm(N, \sigma_i^2, b_N^{(i)})$ as in Definition \ref{defn:rbm} assuming $b_N^{(i)} \gg 1$ for each $i \in I$ \cite{Au18a}, generalizing the result for a single RBM $\#(I) = 1$ \cite{BMP91}.
\end{eg}

We briefly recall the strategy for proving \eqref{eq:nc_wick_formula} based on the graph formalism in Definition \ref{defn:graphs_of_matrices}. Example \ref{eg:moments} reduces the calculation to understanding the asymptotics of $\E[\chi^0(C_p^\pi, \mcal{Z}_N)]$ for each $\pi \in \mcal{P}(\msr{V}(C_p))$. It turns out that only a certain class of graphs survive in the limit, so-called double trees. Before giving the definition, it will be convenient to introduce some notation that will allow us to extract the relevant information from a quotient test graph $T^\pi$.

\begin{defn}[Graph projections]\label{defn:graph_projections}
Let $G = (V, E, \source, \target)$ be a multidigraph. For $\pi \in \mcal{P}(V)$, we define an equivalence relation on the edges $E$ according to the parallel edges of $G^\pi$: 
\[
    e \sim_\pi e' \iff \{[\target(e)]_\pi, [\source(e)]_\pi\} = \{[\target(e')]_\pi, [\source(e')]_\pi\}.
\]
We write $[e]_\pi = \{e' \in E : e' \sim_\pi e \}$ and $[E]_\pi = \{[e]_\pi : e \in E\}$. We separate the projection of loops $[E]_\pi^{(1)} = \{[e]_\pi \in [E]_\pi : [\target(e)]_\pi = [\source(e)]_\pi\}$ from non-loop edges $[E]_\pi^{(2)} = \{[e]_\pi \in [E]_\pi : [\target(e)]_\pi \neq [\source(e)]_\pi\}$. Note that $\underline{G^\pi} := (V^\pi, [E]_\pi^{(2)})$ is the underlying simple graph of $G^\pi$. If $\pi$ is the partition of singletons, then $G^\pi = G$ and we omit it from the notation (e.g., $[e]$ and $[E]$).
\end{defn}

This allows us to formalize the intuitive notion of a double tree.

\begin{defn}[Double tree]\label{defn:double_tree}
A \emph{double tree} is a multidigraph $G = (V, E, \source, \target)$ such that
\begin{enumerate}[label=(\roman*)]
\item there are no loops: $[E]^{(1)} = \emptyset$;
\item every edge is of multiplicity two: $\#([e]) = 2$ for each $e \in E$;
\item the underlying simple graph $\underline{G} = (V, [E])$ is a tree.
\end{enumerate}
The parallel edges of a double tree come in pairs, allowing us to write $[E] = \{\{e, e'\}: e \in E\}$. We say that a test graph $T = (G, \gamma)$ is a \emph{colored double tree} if $G$ is a double tree such that $\gamma(e) = \gamma(e')$ for every pair of parallel edges $\{e, e'\} \in [E]$. 
\end{defn}
Specializing \cite[Lemma 4.3]{Au18a} to quotients of $C_p$, we obtain
\begin{equation}\label{eq:double_tree}
    \lim_{N \to \infty} \E\left[\frac{1}{N}\chi^0(C_p^\pi, \mcal{Z}_N)\right] =
    \begin{dcases}
    \prod_{\{e, e'\} \in [\msr{E}(C_p)]_\pi} \sigma_{\gamma(e)}^2 &\text{if } C_p^\pi \text{ is a colored double tree;} \\
    0 &\text{else,}
    \end{dcases}
\end{equation}
from which \eqref{eq:nc_wick_formula} now follows.

Of course, one can apply the same formalism to $\E[\Tr(p(\mcal{Z}_N) \mbf{x}_N \mbf{y}_N^*)]$, but it will be convenient to separate the contribution from $\mbf{x}_N \mbf{y}_N^*$. 
\begin{eg}[Weighted moments]\label{eg:weighted_moments}
For a monomial $p(\vec{z}) = z_{i(1)} \cdots z_{i(d)} \in \C\langle z_i : i \in I \rangle$ of degree $d$, we define the test graph
\begin{equation}\label{eq:path_test_graph}
  \begin{tikzpicture}[shorten > = 2.5pt, baseline=(current  bounding  box.center)]
    \node at (0, 0) {$T_p = $};
    \draw[fill=black] (.75,0) circle (1.75pt);
    \draw[fill=black] (1.75,0) circle (1.75pt);
    \draw[fill=black] (2.75,0) circle (1.75pt);
    \draw[fill=black] (3.75,0) circle (1.75pt);
    \node at (.75, -.375) {\footnotesize$v_{0}$\normalsize};
    \node at (1.75, -.375) {\footnotesize$v_{1}$\normalsize};
    \node at (2.775, -.375) {\footnotesize$\cdots$\normalsize};
    \node at (3.75, -.375) {\footnotesize$v_{d}$\normalsize};
    \draw[semithick, ->] (1.75,0) to node[midway, above] {\footnotesize$i(1)$\normalsize} (.75,0);
    \draw[semithick, ->] (2.75,0) to (1.75,0);
    \draw[semithick, ->] (3.75,0) to node[midway, above] {\footnotesize$i(d)$\normalsize} (2.75,0);
    \node at (2.375, .28875) {\footnotesize$\cdots$\normalsize};
  \end{tikzpicture}
\end{equation}
Formally, $T_p = (L_p, \gamma_p)$, where $L_p = (V_p, E_p, \source, \target)$ and $\gamma_p : E_p \to I$ satisfy
\begin{enumerate}[label=(\roman*)]
    \item $V_p = \{v_{t-1}: t \in [d + 1]\}$;
    \item $E_p = \{e_{t}: t \in [d]\}$;
    \item $\source(e_t) = v_t$ and $\target(e_t) = v_{t-1}$;
    \item $\gamma_p(e_t) = i(t)$.
\end{enumerate}
We then have the analogue of \eqref{eq:trace}:
\begin{align*}
    \Tr(p(\mcal{M}_N)\mbf{x}_N\mbf{y}_N^*) &= \sum_{\phi: V_p \to [N]} \mbf{x}_N(\phi(v_d)) \overline{\mbf{y}_N(\phi(v_0))} \prod_{e \in E_p} \mbf{M}_N^{(\gamma_p(e))}(\phi(e)) \\
    &= \sum_{\pi \in \mcal{P}(V_p)} \sum_{\phi: V_p^\pi \hookrightarrow [N]} \mbf{x}_N(\phi([v_d]_{\pi})) \overline{\mbf{y}_N(\phi([v_0]_{\pi}))} \prod_{e \in E_p} \mbf{M}_N^{(\gamma_p(e))}(\phi(e)), 
\end{align*}
where we recall that $E_p^\pi = E_p$.
\end{eg}

\begin{rem}\label{rem:intuition}
Roughly speaking, the isotropic global law amounts to the asymptotic 
\begin{align*}
    \E[\Tr(p(\mcal{Z}_N)\mbf{x}_N\mbf{y}_N^*)] &= \sum_{\pi \in \mcal{P}(V_p)} \sum_{\phi: V_p^\pi \hookrightarrow [N]} \mbf{x}_N(\phi([v_d]_{\pi})) \overline{\mbf{y}_N(\phi([v_0]_{\pi}))} \prod_{e \in E_p} \mbf{\Xi}_N^{(\gamma_p(e))}(\phi(e)) \\
    &= \sum_{\substack{\pi \in \mcal{P}(V_p) \\ \text{s.t. }v_0 \stackrel{\pi}{\sim} v_d}} \sum_{\phi: V_p^\pi \hookrightarrow [N]} \mbf{x}_N(\phi([v_d]_{\pi})) \overline{\mbf{y}_N(\phi([v_0]_{\pi}))} \prod_{e \in E_p} \mbf{\Xi}_N^{(\gamma_p(e))}(\phi(e)) + o(1).
\end{align*}
To see this, note that the cycle graph in \eqref{eq:moment_graph} and the path graph in \eqref{eq:path_test_graph} satisfy $C_p = T_p^{\hat{\pi}}$ for the partition $\hat{\pi}$ whose only nonsingleton block is $\{v_0, v_d\}$. By restricting to the class of partitions $\pi \in \mcal{P}(V_p)$ such that $v_0 \stackrel{\pi}{\sim} v_d$, the set of possible quotients of $T_p$ is then equal to the set of possible quotients of $C_p$. The identification $[v_0]_{\pi} = [v_d]_{\pi}$ also forces $\phi([v_0]_{\pi}) = \phi([v_d]_{\pi})$, which both explains the inner product and introduces the additional normalization that is seemingly missing compared to \eqref{eq:nc_wick_formula}.
\end{rem}

\section{Proofs of the main results}\label{sec:proofs}

\subsection{The isotropic global law}\label{sec:isotropic_global_law}

Let $\mcal{Z}_N = (\mbf{\Xi}_N^{(i)})_{i \in I}$ be a family of independent RBMs $\mbf{\Xi}_N^{(i)} \deq \rbm(N, \sigma_i^2, b_N^{(i)})$ as in Definition \ref{defn:rbm}. We assume that $b_N^{(i)} \gg 1$ for each $i \in I$ to ensure the convergence in \eqref{eq:nc_wick_formula}, which also holds almost surely \cite[Theorem 4.12]{Au18a}:
\[
    \lim_{N \to \infty} \frac{1}{N}\Tr(p(\mcal{Z}_N)) \aseq \sum_{\pi \in \mcal{NC}_2(d)} \prod_{\{j, k\} \in \pi} \sigma_{i(j)}\sigma_{i(k)} \indc{i(j) = i(k)} =: \tau(p).
\]
The main technical contribution of this article is an isotropic version of this convergence. To state the precise result, we need some additional notation. For $p_1(\vec{z}), \ldots, p_r(\vec{z}) \in \C\langle z_i : i \in I \rangle$, we define
\[
    I_{p_1, \ldots, p_r} := \{i \in I : z_i \text{ appears in } p_s(\vec{z}) \text{ for some } s \in [r]\}.
\]
We start by proving convergence in expectation (cf.\@ \cite[Lemma 3.4]{Au21}).

\begin{lemma}[Isotropic global law, in expectation]\label{lem:isotropic_global_law_expectation}
For $p(\vec{z}) \in \C\langle z_i : i \in I \rangle$ and $\mbf{x}_N, \mbf{y}_N \in \sphr^{N-1}$,
\[
    \E\left[\Tr(p(\mcal{Z}_N) \mbf{x}_N {\mbf{y}_N}^*)\right] = \langle \mbf{x}_N, \mbf{y}_N \rangle \tau(p) + O_p\left(\frac{1}{\min_{i \in I_p} \sqrt{\xi_N^{(i)}}}\right).
\]
In particular, the constant in the asymptotic does not depend on the unit vectors $\mbf{x}_N, \mbf{y}_N$.
\end{lemma}
\begin{proof}
By linearity, we may assume that $p(\vec{z}) = z_{i(1)} \cdots z_{i(d)}$ is a monomial. To simplify the notation, we abbreviate the test graph $T_p = (L_p, \gamma_p)$ in \eqref{eq:path_test_graph} to $T = (L, \gamma)$. The trace can then be expanded using the graph formalism in Example \ref{eg:weighted_moments}:\small
\begin{align}
    \E\left[\Tr(p(\mcal{Z}_N)\mbf{x}_N\mbf{y}_N^*)\right]
    &= \sum_{\phi: V \to [N]} \mbf{x}_N(\phi(v_d)) \overline{\mbf{y}_N(\phi(v_0))} \E\left[\prod_{e \in E} \mbf{\Xi}_N^{(\gamma(e))}(\phi(e))\right] \notag \\
    &= \sum_{\pi \in \mcal{P}(V)} \sum_{\phi: V^\pi \hookrightarrow [N]} \mbf{x}_N(\phi([v_d]_{\pi})) \overline{\mbf{y}_N(\phi([v_0]_{\pi}))} \E\left[\prod_{e \in E} \mbf{X}_N^{(\gamma(e))}(\phi(e))\right] \frac{\prod_{e \in E} \mbf{B}_N^{(\gamma(e))}(\phi(e))}{\prod_{e \in E} \sqrt{\xi_N^{(\gamma(e))}}} \notag \\ 
    &=: \sum_{\pi \in \mcal{P}(V)} \sum_{\phi: V^\pi \hookrightarrow [N]} \zeta_N(\pi, \phi),
    \label{eq:expanded_trace}
\end{align}\normalsize
where
\begin{equation}\label{eq:expectation_over_edges}
    \E\left[\prod_{e \in E} \mbf{X}_N^{(\gamma(e))}(\phi(e))\right] = O_p(1)
\end{equation}
uniformly in $(\pi, \phi)$ by our moment assumption \eqref{eq:finite_moments} and the finiteness of $\gamma(E) \subset I$. Since $\phi: V^\pi \hookrightarrow [N]$ is injective, the independence of our random variables allows us to factor the expectation over parallel edges. In particular, using the notation in Definition \ref{defn:graph_projections},
\[
    \E\left[\prod_{e \in E} \mbf{X}_N^{(\gamma(e))}(\phi(e))\right] = \prod_{l=1}^2 \prod_{[e]_\pi \in [E]_\pi^{(l)}} \E\left[\prod_{e' \in [e]_\pi} \mbf{X}_N^{(\gamma(e'))}(\phi(e')) \right].
\]
The centeredness of the off-diagonal random variables tells us that 
\begin{equation}\label{eq:zero_term}
    \prod_{[e]_\pi \in [E]_\pi^{(2)}} \E\left[\prod_{e' \in [e]_\pi} \mbf{X}_N^{(\gamma(e'))}(\phi(e')) \right] = 0
\end{equation}
unless $\#([e]_\pi) \geq 2$ for every $[e]_\pi \in [E]_\pi^{(2)}$. So, we may restrict the outer sum in \eqref{eq:expanded_trace} to such partitions. This leads us to define
\begin{equation}\label{eq:restricted_partitions}
    \mcal{P}_+(V) := \{\pi \in \mcal{P}(V) : \#([e]_\pi) \geq 2  \text{ for every } [e]_\pi \in [E]_\pi^{(2)}\},
\end{equation}
which allows us to rewrite \eqref{eq:expanded_trace} as
\[
    \E\left[\Tr(p(\mcal{Z}_N)\mbf{x}_N\mbf{y}_N^*)\right] = \sum_{\pi \in \mcal{P}_+(V)} \sum_{\phi: V^\pi \hookrightarrow [N]} \zeta_N(\pi, \phi).
\]

We introduce some additional notation to control the inner sum. Recall that $\underline{L^\pi} = (V^\pi, [E]_\pi^{(2)})$ is the underlying simple graph of $L^\pi$. Let $(V^\pi, [F]_\pi)$ be a spanning tree of $\underline{L^\pi}$. For $e \in E$, we define
\[
    \lfloor e \rfloor_\pi := \argmin_{e' \in [e]_\pi} b_N^{(\gamma(e'))} = \argmin_{e' \in [e]_\pi} \xi_N^{(\gamma(e'))}.
\]
In the event of a tie, we choose the leftmost edge in the path $L$ for concreteness. Since the matrices $(\mbf{B}_N^{(i)})_{i \in I}$ are symmetric $(0, 1)$-matrices of the form \eqref{eq:band_matrix}, we can bound the contribution from the product
\begin{align*}
    \prod_{e \in E} \mbf{B}_N^{(\gamma(e))}(\phi(e)) &= \prod_{l=1}^2 \prod_{[e]_\pi \in [E]_\pi^{(l)}} \prod_{e' \in [e]_\pi} \mbf{B}_N^{(\gamma(e'))}(\phi(e')) \\
    &= \prod_{[e]_\pi \in [E]_\pi^{(2)}} \prod_{e' \in [e]_\pi} \mbf{B}_N^{(\gamma(e'))}(\phi(e')) \\
    &= \prod_{[e]_\pi \in [E]_\pi^{(2)}} \mbf{B}_N^{(\gamma(\lfloor e \rfloor_\pi))}(\phi(\lfloor e \rfloor_\pi)) \\
    &\leq \prod_{[e]_\pi \in [F]_\pi} \mbf{B}_N^{(\gamma(\lfloor e \rfloor_\pi))}(\phi(\lfloor e \rfloor_\pi)),
\end{align*}
where the symmetry eliminates the ambiguity in the direction of the edge $\lfloor e \rfloor_\pi$ for the purposes of $\mbf{B}_N^{(\gamma(\lfloor e \rfloor_\pi))}(\phi(\lfloor e \rfloor_\pi))$. Combining this with the bound on the expectation \eqref{eq:expectation_over_edges}, we obtain
\begin{align*}
    \zeta_N(\pi, \phi) &= O_p\left(|\mbf{x}_N(\phi([v_d]_{\pi}))| |\mbf{y}_N(\phi([v_0]_{\pi}))| \frac{\prod_{[e]_\pi \in [F]_\pi} \mbf{B}_N^{(\gamma(\lfloor e \rfloor_\pi))}(\phi(\lfloor e \rfloor_\pi))}{\prod_{e \in E} \sqrt{\xi_N^{(\gamma(e))}}}\right).
\end{align*}
This allows us to further restrict to partitions
\[
    \mcal{P}_{++}(V) := \{\pi \in \mcal{P}(V) : T^\pi \text{ is a colored double tree}\} \subset \mcal{P}_+(V)
\]
at the cost of
\begin{equation}\label{eq:restricted_sum}
    \sum_{\pi \in \mcal{P}_+(V)} \sum_{\phi: V^\pi \hookrightarrow [N]} \zeta_N(\pi, \phi) = \sum_{\pi \in \mcal{P}_{++}(V)} \sum_{\phi: V^\pi \hookrightarrow [N]} \zeta_N(\pi, \phi) + O\left(\frac{1}{\min_{e \in E} \sqrt{\xi_N^{(\gamma(e))}}}\right).
\end{equation}
Before proving this, note that if $T^\pi$ is a double tree, then $[v_0]_\pi = [v_d]_\pi$. Indeed, every vertex in a double tree has even degree; however, every vertex $v \not\in \{v_0, v_d\}$ in $T$ has degree two, while $\deg_{T}(v_0) = \deg_{T}(v_d) = 1$. Since $\deg_{T^\pi}([v]_\pi) = \sum_{v' \sim_\pi v} \deg_{T}(v')$, the result follows. 

To prove \eqref{eq:restricted_sum}, assume that $\pi \in \mcal{P}_+(V)\setminus\mcal{P}_{++}(V)$. We consider two cases: $[v_0]_\pi \neq [v_d]_\pi$ and $[v_0]_\pi = [v_d]_\pi$. If $[v_0]_\pi \neq [v_d]_\pi$, then there is a unique (necessarily nonempty) path from $[v_0]_\pi$ to $[v_d]_\pi$ in $(V^\pi, [F]_\pi)$. We enumerate the edges $[f_1]_\pi, \ldots, [f_m]_\pi$ on this path as well as the vertices $w_0, \ldots, w_m$. We separate the remaining edges $[F']_\pi = [F]_\pi \setminus \{[f_1]_\pi, \ldots, [f_m]_\pi\}$ to bound
\begin{align*}
    &\sum_{\phi: V^\pi \hookrightarrow [N]}  |\mbf{x}_N(\phi([v_d]_{\pi}))| |\mbf{y}_N(\phi([v_0]_{\pi}))| \frac{\prod_{[e]_\pi \in [F]_\pi} \mbf{B}_N^{(\gamma(\lfloor e \rfloor_\pi))}(\phi(\lfloor e \rfloor_\pi))}{\prod_{e \in E} \sqrt{\xi_N^{(\gamma(e))}}} \\
    \leq &\sum_{\Phi: \{w_n\}_{n = 0}^m \to [N]} |\mbf{x}_N(\Phi(w_m))| |\mbf{y}_N(\Phi(w_0))| \prod_{n = 1}^m \mbf{B}_N^{(\gamma(\lfloor f_n \rfloor_\pi))}(\Phi(w_{n -1}), \Phi(w_n)) \\
    \phantom{=} &\phantom{\sum_{\phi: \{w_n\}_{n = 0}^m \to [N]} |\mbf{x}_N(\phi(w_m))|} \cdot \frac{\sum_{\Psi: V^\pi \setminus \{w_n\}_{n = 0}^m \to [N]} \prod_{[e]_\pi \in [F']_\pi} \mbf{B}_N^{(\gamma(\lfloor e \rfloor_\pi))}(\Psi(\lfloor e \rfloor_\pi))}{\prod_{e \in E} \sqrt{\xi_N^{(\gamma(e))}}} \\
    = &\frac{\prod_{[e]_\pi \in [F']_\pi} \xi_N^{(\gamma(\lfloor e \rfloor_\pi))}}{\prod_{e \in E} \sqrt{\xi_N^{(\gamma(e))}}} \sum_{\Phi: \{w_n\}_{n=0}^m \to [N]} |\mbf{x}_N(\Phi(w_m))| |\mbf{y}_N(\Phi(w_0))| \prod_{n = 1}^m \mbf{B}_N^{(\gamma(\lfloor f_n \rfloor_\pi))}(\Phi(w_{n -1}), \Phi(w_n)),
\end{align*}
where we have again used the fact that the matrices $(\mbf{B}_N^{(i)})_{i \in I}$ are symmetric to ignore the directions of the edges in writing $\mbf{B}_N^{(\gamma(\lfloor f_n \rfloor_\pi))}(\Phi(\lfloor f_n \rfloor_\pi)) = \mbf{B}_N^{(\gamma(\lfloor f_n \rfloor_\pi))}(\Phi(w_{n -1}), \Phi(w_n))$. We recognize the remaining sum as an inner product \small
\[
    \sum_{\Phi: \{w_0\}_{n=1}^m \to [N]} |\mbf{x}_N(\Phi(w_m))| |\mbf{y}_N(\Phi(w_0))| \prod_{n = 1}^m \mbf{B}_N^{(\gamma(\lfloor f_n \rfloor_\pi))}(\Phi(w_{n -1}), \Phi(w_n))
    = \left\langle \prod_{n = 1}^m \mbf{B}_N^{(\gamma(\lfloor f_n \rfloor_\pi))} |\mbf{x}_N|, |\mbf{y}_N|\right\rangle,
\]\normalsize
where $|\mbf{x}_N|, |\mbf{y}_N|$ are the vectors obtained from $\mbf{x}_N, \mbf{y}_N$ by applying the entrywise absolute value. In particular, it is still the case that $|\mbf{x}_N|, |\mbf{y}_N| \in \sphr^{N-1}$. Thus,
\begin{align*}
    \left\langle \prod_{n = 1}^m \mbf{B}_N^{(\gamma(\lfloor f_n \rfloor_\pi))} |\mbf{x}_N|, |\mbf{y}_N|\right\rangle
    &\leq \snorm{\prod_{n = 1}^m \mbf{B}_N^{(\gamma(\lfloor f_n \rfloor_\pi))}}_2 \\
    &\leq \prod_{n = 1}^m \snorm{\mbf{B}_N^{(\gamma(\lfloor f_n \rfloor_\pi))}}_2 = \prod_{n = 1}^m \xi_N^{(\gamma(\lfloor f_n \rfloor_\pi))},
\end{align*}
where the operator norm calculation follows from the observation that $\mbf{B}_N^{(i)}$ is a real symmetric $(0,1)$-matrix with every row sum equal to $\xi_N^{(i)}$ (see, for example, \cite[Problem 5.6.P21]{HJ13}).

Putting everything together, we have that
\begin{align*}
    \sum_{\phi: V^\pi \hookrightarrow [N]} \zeta_N(\pi, \phi) &= O_p\left(\frac{\prod_{[e]_\pi \in [F']_\pi} \xi_N^{(\gamma(\lfloor e \rfloor_\pi))} \prod_{n = 1}^m \xi_N^{(\gamma(\lfloor f_m \rfloor_\pi))}}{\prod_{e \in E} \sqrt{\xi_N^{(\gamma(e))}}}\right) \\
    &= O_p\left(\frac{\prod_{[e]_\pi \in [F]_\pi} \xi_N^{(\gamma(\lfloor e \rfloor_\pi))}}{\prod_{[e]_\pi \in [E]_\pi} \prod_{e' \in [e]_\pi} \sqrt{\xi_N^{(\gamma(e'))}}}\right).
\end{align*}
Since $\pi \in \mcal{P}_+(V) \setminus \mcal{P}_{++}(V)$, we know that
\begin{equation}\label{eq:normalization_balance}
    \frac{\xi_N^{(\gamma(\lfloor e \rfloor_\pi))}}{\prod_{e' \in [e]_\pi} \sqrt{\xi_N^{(\gamma(e'))}}} \leq 1, \qquad \forall [e]_\pi \in [E]_\pi^{(2)}.
\end{equation}
Furthermore, by definition, $T^\pi$ is not a colored double tree. If $T^\pi$ is a (miscolored) double tree, then we are done since the independence and centeredness of the off-diagonal entries would again imply \eqref{eq:zero_term}. So, we may assume that $T^\pi$ is not a double tree. This means that either the underlying simple graph is not a tree, in which case $[F]_\pi \subsetneq [E]_\pi$, or the underlying simple graph is a tree, but there is at least one edge $[e]_\pi$ with multiplicity $\#([e]_\pi) \geq 3$. In either case, we see that
\begin{equation}\label{eq:non_double_tree}
    \frac{\prod_{[e]_\pi \in [F]_\pi} \xi_N^{(\gamma(\lfloor e \rfloor_\pi))}}{\prod_{[e]_\pi \in [E]_\pi} \prod_{e' \in [e]_\pi} \sqrt{\xi_N^{(\gamma(e'))}}} = O_p\left(\frac{1}{\min_{e \in E} \sqrt{\xi_N^{(\gamma(e))}}}\right).
\end{equation}

The remaining case of $\pi \in \mcal{P}_+(V)\setminus\mcal{P}_{++}(V)$ such that $[v_0]_\pi = [v_d]_\pi$ is treated much the same. Indeed, the Cauchy-Schwarz inequality tells us that
\begin{align*}
    &\sum_{\phi: V^\pi \hookrightarrow [N]} |\mbf{x}_N(\phi([v_d]_{\pi}))| |\mbf{y}_N(\phi([v_0]_{\pi}))| \frac{\prod_{[e]_\pi \in [F]_\pi} \mbf{B}_N^{(\gamma(\lfloor e \rfloor_\pi))}(\phi(\lfloor e \rfloor_\pi))}{\prod_{e \in E} \sqrt{\xi_N^{(\gamma(e))}}} \\
    \leq &\frac{\prod_{[e]_\pi \in [F]} \xi_N^{(\gamma(\lfloor e \rfloor_\pi))}}{\prod_{[e]_\pi \in [E]_\pi} \prod_{e' \in [e]_\pi} \sqrt{\xi_N^{(\gamma(e'))}}} \sum_{i = 1}^N |\mbf{x}_N(i)| |\mbf{y}_N(i)| \\
    \leq  &\frac{\prod_{[e]_\pi \in [F]} \xi_N^{(\gamma(\lfloor e \rfloor_\pi))}}{\prod_{[e]_\pi \in [E]_\pi} \prod_{e' \in [e]_\pi} \sqrt{\xi_N^{(\gamma(e'))}}}.
\end{align*}
As before, we can assume that $T^\pi$ is not a double tree, which again leads to the asymptotic \eqref{eq:non_double_tree}. We conclude that
\[
    \sum_{\pi \in \mcal{P}_+(V)\setminus\mcal{P}_{++}(V)} \sum_{\phi: V^\pi \hookrightarrow [N]} \zeta_N(\pi, \phi) = O_p\left(\frac{1}{\min_{e \in E} \sqrt{\xi_N^{(\gamma(e))}}}\right),
\]
which proves \eqref{eq:restricted_sum}.

To finish the proof, consider a partition $\pi \in \mcal{P}_{++}(V)$. By definition, $T^\pi$ is a colored double tree such that $[v_0]_\pi = [v_d]_\pi$. We can think of performing the identification $v_0 \stackrel{\pi}{\sim} v_d$ first and view $T^\pi$ as a quotient of the directed cycle $C_p = T^{\hat{\pi}}$ in \eqref{eq:moment_graph}, where the only nonsingleton block in $\hat{\pi}$ is $\{v_0, v_d\}$. It is not hard to see that a quotient of a directed cycle is a double tree only if each of its parallel edges $\{e, e'\}$ point in opposite directions $(\target(e), \source(e)) = (\source(e'), \target(e'))$ \cite[Figure 5]{Au18a}. Thus, the expectation in $\zeta_N(\pi, \phi)$ can be computed entirely in terms of the variances:
\small
\begin{align*}
    \sum_{\phi: V^\pi \hookrightarrow [N]} \zeta_N(\pi, \phi) &= \sum_{\phi: V^\pi \hookrightarrow [N]} \mbf{x}_N(\phi([v_d]_{\pi})) \overline{\mbf{y}_N(\phi([v_0]_{\pi}))} \E\left[\prod_{e \in E} \mbf{X}_N^{(\gamma(e))}(\phi(e))\right] \frac{\prod_{e \in E} \mbf{B}_N^{(\gamma(e))}(\phi(e))}{\prod_{e \in E} \sqrt{\xi_N^{(\gamma(e))}}} \\
    &= \left[\prod_{[e] \in [E]_\pi} \sigma_{\gamma(e)}^2\right] \sum_{\Phi: \{[v_0]_\pi\} \to [N]} \mbf{x}_N(\Phi([v_0]_\pi)) \overline{\mbf{y}_N(\Phi([v_0]_\pi))} \\
    &\phantom{=\left[\prod_{[e] \in [E]_\pi} \sigma_{\gamma(e)}^2\right] \sum_{\Phi: \{[v_0]_\pi\} \to [N]}} \cdot\left(\frac{\sum_{\Psi: V^\pi \setminus \{[v_0]_\pi\} \hookrightarrow [N]\setminus\{\Phi([v_0]_\pi)\}} \prod_{[e]_\pi \in [E]_\pi} \mbf{B}_N^{(\gamma(e))}(\Psi(e))}{\prod_{[e] \in [E]_\pi} \xi_N^{(\gamma(e))}}\right).
\end{align*}\normalsize
As before, we can bound the number of maps $\Psi$ that will produce a nonzero summand (necessarily equal to 1) by
\[
    \sum_{\Psi: V^\pi \setminus \{[v_0]_\pi\} \hookrightarrow [N]\setminus\{\Phi([v_0]_\pi)\}} \prod_{[e]_\pi \in [E]_\pi} \mbf{B}_N^{(\gamma(e))}(\Psi(e)) \leq \prod_{[e]_\pi \in [E]_\pi} \xi_N^{(\gamma(e))}.
\]
Since $(V^\pi, [E]_\pi)$ is a tree, the only obstruction to equality is the required injectivity of $\Psi$. This gives the straightforward lower bound
\[
    \sum_{\Psi: V^\pi \setminus \{[v_0]_\pi\} \hookrightarrow [N]\setminus\{\Phi([v_0]_\pi)\}} \prod_{[e]_\pi \in [E]_\pi} \mbf{B}_N^{(\gamma(e))}(\Psi(e)) \geq \prod_{[e]_\pi \in [E]_\pi} (\xi_N^{(\gamma(e))} - d),
\]
where we recall that $d \geq \#(V^\pi) - 1$ is the degree of the monomial $p(\vec{z})$. We conclude that
\begin{align*}
    &\left[\prod_{[e] \in [E]_\pi} \sigma_{\gamma(e)}^2\right] \sum_{\Phi: \{[v_0]_\pi\} \to [N]} \mbf{x}_N(\Phi([v_0]_\pi)) \overline{\mbf{y}_N(\Phi([v_0]_\pi))} \\
    &\phantom{=\left[\prod_{[e] \in [E]_\pi} \sigma_{\gamma(e)}^2\right] \sum_{\Phi: \{[v_0]_\pi\} \to [N]}} \cdot\left(\frac{\sum_{\Psi: V^\pi \setminus \{[v_0]_\pi\} \hookrightarrow [N]\setminus\{\Phi([v_0]_\pi)\}} \prod_{[e]_\pi \in [E]_\pi} \mbf{B}_N^{(\gamma(e))}(\Psi(e))}{\prod_{[e] \in [E]_\pi} \xi_N^{(\gamma(e))}}\right) \\
    = &\langle \mbf{x}_N, \mbf{y}_N \rangle \prod_{[e] \in [E]_\pi} \sigma_{\gamma(e)}^2 + O_p\left(\frac{1}{\min_{e \in E} \xi_N^{(\gamma(e))}}\right).
\end{align*}
In view of the usual calculation for the normalized trace \eqref{eq:double_tree}, we are done.
\end{proof}

\begin{rem}\label{rem:expectation_k_regular_wigner}
In the case of independent $(k_N^{(i)})_{i \in I}$-regular Wigner matrices
\[
    (\mbf{\Xi}_N^{(i)})_{i \in I} = \left(\frac{1}{\sqrt{k_N^{(i)}}}\wtilde{\mbf{B}}_N^{(i)} \circ \mbf{X}_N^{(i)}\right)_{i \in I},
\]
we simply need to replace all instances of $\mbf{B}_N^{(i)}$ (resp., $\xi_N^{(i)}$) with $\wtilde{\mbf{B}}_N^{(i)}$ (resp., $k_N^{(i)}$) in the proof with one notable exception. In particular, for periodic $(0,1)$-band matrices $(\mbf{B}_N^{(i)})_{i \in I}$, we repeatedly used the identity 
\[
    \circ_{e' \in [e]_\pi} \mbf{B}_N^{(\gamma(e'))} = \mbf{B}_N^{(\gamma(\lfloor e \rfloor_\pi))},
\]
where we recall that $\circ$ denotes the entrywise product. While this no longer holds in general for $(\wtilde{\mbf{B}}_N^{(i)})_{i \in I}$, it is true that $\circ_{e' \in [e]_\pi} \wtilde{\mbf{B}}_N^{(\gamma(e'))}$ is a symmetric $(0, 1)$-matrix with row sums bounded by $k_N^{(\gamma(\lfloor e \rfloor_\pi))} := \min_{e' \in [e]_\pi} k_N^{(\gamma(e'))}$, which is all that is needed to carry forward the same argument.
\end{rem}

Having computed the expectation, we proceed to proving concentration. For this, we use a bound on central moments.

\begin{lemma}[Central moments]\label{lem:central_moments}
For $(p_s(\vec{z}))_{s = 1}^r \subset \C\langle z_i : i \in I \rangle$ and $(\mbf{x}_N^{(s)})_{s = 1}^r, (\mbf{y}_N^{(s)})_{s = 1}^r \subset \sphr^{N-1}$,
\[
    \E\left[\prod_{s = 1}^r \left(\Tr\left(p_s(\mcal{Z}_N)\mbf{x}_N^{(s)} {\mbf{y}_N^{(s)}}^*\right) - \E\left[\Tr\left(p_s(\mcal{Z}_N)\mbf{x}_N^{(s)}{\mbf{y}_N^{(s)}}^*\right) \right] \right) \right] = O_{p_1, \ldots, p_r}\left(\left[\min_{i \in I_{p_1, \ldots, p_r}} \sqrt{\xi_N^{(i)}}\right]^{-r}\right).
\]
As before, the constant in the asymptotic does not depend on the unit vectors $(\mbf{x}_N^{(s)})_{s = 1}^r, (\mbf{y}_N^{(s)})_{s = 1}^r$.
\end{lemma}

\begin{proof}
By multilinearity, we may assume that each $p_s(\vec{z}) = z_{i_s(1)} \cdots z_{i_s(d_s)}$ is a monomial. To simplify the notation, we abbreviate the test graph $T_{p_s} = (L_{p_s}, \gamma_{p_s})$ in \eqref{eq:path_test_graph} to $T_s = (L_s, \gamma_s)$. We also define $T = (G, \gamma)$ to be the disjoint union $T = \sqcup_{s = 1}^r T_s$ of the test graphs $(T_s)_{s = 1}^r$. The analogue of \eqref{eq:expanded_trace} for central moments can then be written as \small
\begin{align*}
    &\E\left[\prod_{s = 1}^r \left(\Tr\left(p_s(\mcal{Z}_N)\mbf{x}_N^{(s)} {\mbf{y}_N^{(s)}}^*\right) - \E\left[\Tr\left(p_s(\mcal{Z}_N)\mbf{x}_N^{(s)}{\mbf{y}_N^{(s)}}^*\right) \right] \right) \right] \\
    = &\E\left[\prod_{s = 1}^r \left( \sum_{\phi_s: V_s \to [N]} \mbf{x}_N^{(s)}(\phi_s(v_{d_s}^{(s)})) \overline{\mbf{y}_N^{(s)}(\phi_s(v_0^{(s)}))} \left(\prod_{e \in E_s} \mbf{\Xi}_N^{(\gamma_s(e))}(\phi_s(e)) - \E\left[\prod_{e \in E_s} \mbf{\Xi}_N^{(\gamma_s(e))}(\phi_s(e))\right] \right) \right) \right] \\
    = &\sum_{\phi: V \to [N]} \left(\prod_{s = 1}^r \mbf{x}_N^{(s)}(\phi(v_{d_s}^{(s)})) \overline{\mbf{y}_N^{(s)}(\phi(v_0^{(s)}))}\right) \E\left[\prod_{s = 1}^r \left(\prod_{e \in E_s} \mbf{\Xi}_N^{(\gamma(e))}(\phi(e)) - \E\left[\prod_{e \in E_s} \mbf{\Xi}_N^{(\gamma(e))}(\phi(e))\right] \right) \right] \\
    = &\sum_{\pi \in \mcal{P}(V)} \sum_{\phi: V^\pi \hookrightarrow [N]} \left(\prod_{s = 1}^r \mbf{x}_N^{(s)}(\phi([v_{d_s}^{(s)}]_\pi)) \overline{\mbf{y}_N^{(s)}(\phi([v_0^{(s)}]_\pi))}\right) \\
    \phantom{=} &\phantom{\sum_{\pi \in \mcal{P}(V) \prod_{s = 1}^r} \sum_{\phi: V^\pi \hookrightarrow [N]}} \cdot \E\left[\prod_{s = 1}^r  \left(\prod_{e \in E_s} \mbf{X}_N^{(\gamma(e))}(\phi(e)) - \E\left[\prod_{e \in E_s} \mbf{X}_N^{(\gamma(e))}(\phi(e))\right]\right) \right] \frac{\prod_{e \in E} \mbf{B}_N^{(\gamma(e))}(\phi(e))}{\prod_{e \in E} \sqrt{\xi_N^{(\gamma(e))}}} \\
    =: &\sum_{\pi \in \mcal{P}(V)} \sum_{\phi: V^\pi \hookrightarrow [N]} \eta_N(\phi, \pi). 
\end{align*} \normalsize

We repeat two of the early steps in the proof of Lemma \ref{lem:isotropic_global_law_expectation}. In particular, our moment assumption \eqref{eq:finite_moments} implies that
\[
     \E\left[\prod_{s = 1}^r \left(\prod_{e \in E_s} \mbf{X}_N^{(\gamma(e))}(\phi(e)) - \E\left[\prod_{e \in E_s} \mbf{X}_N^{(\gamma(e))}(\phi(e))\right] \right) \right] = O_{p_1, \ldots, p_r}(1)
\]
uniformly in $(\pi, \phi)$ with
\[
    \E\left[\prod_{s = 1}^r \left(\prod_{e \in E_s} \mbf{X}_N^{(\gamma(e))}(\phi(e)) - \E\left[\prod_{e \in E_s} \mbf{X}_N^{(\gamma(e))}(\phi(e))\right] \right) \right] = 0
\]
unless $\#([e]_\pi) \geq 2$ for every $e \in [E]_\pi^{(2)}$. Since we are considering central moments, we can say even more. In particular, we say that $T_s$ and $T_{s'}$ have an \emph{edge overlay in $T^\pi$} if there exist edges $e_s \in E_s$ and $e_{s'} \in E_{s'}$ such that $[e_s]_\pi = [e_{s'}]_\pi$. The edge overlays define a natural equivalence relation $\sim_{E, \pi}$ on $[r]$ as follows:
\[
     s \sim_{E, \pi} s' \iff \exists s_0, \ldots, s_n \in [r] : T_{s_{t-1}} \text{ and } T_{s_{t}} \text{ have an edge overlay in } T^\pi \text{ for every } t \in [n],
\]
where $s_0 = s$ and $s_n = s'$. We use the notation $[s]_{E, \pi} = \{s' \in [r]: s' \sim_{E, \pi} s\} \in [r]/\sim_{E, \pi}$ to avoid confusion with $[v]_\pi \in V^\pi$ and $[e]_\pi \in [E]_\pi$. This allows us to factor
\begin{align*}
    &\E\left[\prod_{s = 1}^r \left(\prod_{e \in E_s} \mbf{X}_N^{(\gamma(e))}(\phi(e)) - \E\left[\prod_{e \in E_s} \mbf{X}_N^{(\gamma(e))}(\phi(e))\right] \right) \right] \\
    = & \prod_{[s]_{E, \pi} \in [r]/\sim_{E, \pi}} \E\left[\prod_{s' \in [s]_{E, \pi}} \left(\prod_{e \in E_{s'}} \mbf{X}_N^{(\gamma(e))}(\phi(e)) - \E\left[\prod_{e \in E_{s'}} \mbf{X}_N^{(\gamma(e))}(\phi(e))\right] \right) \right], 
\end{align*}
which is equal to $0$ by the centering unless $\#([s]_{E, \pi}) \geq 2$ for every $s \in [r]$. In other words, the expectation vanishes unless every test graph $T_s$ has an edge overlay in $T^\pi$ with at least one other test graph $T_{s'}$. This leads us to define 
\[
    \mcal{P}_\times(V) := \{\pi \in \mcal{P}_+(V): \#([s]_{E, \pi}) \geq 2 \text{ for every } s \in [r]\},
\]
where we recall the definition of $\mcal{P}_+(V)$ in \eqref{eq:restricted_partitions}. The central moment calculation then reduces to 
\begin{equation*}
    \E\left[\prod_{s = 1}^r \left(\Tr\left(p_s(\mcal{Z}_N)\mbf{x}_N^{(s)} {\mbf{y}_N^{(s)}}^*\right) - \E\left[\Tr\left(p_s(\mcal{Z}_N)\mbf{x}_N^{(s)}{\mbf{y}_N^{(s)}}^*\right) \right] \right) \right] = \sum_{\pi \in \mcal{P}_{\times}(V)} \sum_{\phi: V^\pi \hookrightarrow [N]} \eta_N(\phi, \pi),
\end{equation*}
where
\[
    \eta_N(\phi, \pi) = O_{p_1, \ldots, p_r}  \left( \left[\prod_{s = 1}^r \left|\mbf{x}_N^{(s)}(\phi([v_{d_s}^{(s)}]_\pi))\right| \left|\mbf{y}_N^{(s)}(\phi([v_0^{(s)}]_\pi))\right|\right] \frac{\prod_{e \in E} \mbf{B}_N^{(\gamma(e))}(\phi(e))}{\prod_{e \in E} \sqrt{\xi_N^{(\gamma(e))}}} \right).
\] 

The equivalence relation $\sim_{[E]_\pi}$ is not necessarily equal to the equivalence relation on $[r]$ defined by the connected components of $T^\pi = (\sqcup_{s = 1}^r T_s)^\pi$. In particular, if 
\[
    \msr{V}([s]_{E, \pi}) := \{[v]_\pi \in V^\pi: v \in T_{s'} \text{ for some } s' \sim_{E, \pi} s\},
\]
then it could be that $\msr{V}([s]_{E, \pi}) \cap \msr{V}([s']_{E, \pi}) \neq \emptyset$ for $[s]_{E, \pi} \neq [s']_{E, \pi}$. Nevertheless, it is true that
\begin{align*}
    &\left[\prod_{s = 1}^r \left|\mbf{x}_N^{(s)}(\phi([v_{d_s}^{(s)}]_\pi))\right| \left|\mbf{y}_N^{(s)}(\phi([v_0^{(s)}]_\pi))\right|\right] \frac{\prod_{e \in E} \mbf{B}_N^{(\gamma(e))}(\phi(e))}{\prod_{e \in E} \sqrt{\xi_N^{(\gamma(e))}}} \\
    = &\prod_{[s]_{E, \pi} \in [r]/\sim_{E, \pi}} \left( \prod_{s' \in [s]_{E, \pi}} \left[ \left|\mbf{x}_N^{(s')}(\phi([v_{d_{s'}}^{(s')}]_\pi))\right| \left|\mbf{y}_N^{(s')}(\phi([v_0^{(s')}]_\pi))\right| \frac{\prod_{e \in E_{s'}} \mbf{B}_N^{(\gamma(e))}(\phi(e))}{\prod_{e \in E_{s'}} \sqrt{\xi_N^{(\gamma(e))}}} \right] \right),
\end{align*}
whence 
\begin{align*}
    &\sum_{\phi: V^\pi \hookrightarrow [N]} \left[\prod_{s = 1}^r \left|\mbf{x}_N^{(s)}(\phi([v_{d_s}^{(s)}]_\pi))\right| \left|\mbf{y}_N^{(s)}(\phi([v_0^{(s)}]_\pi))\right|\right] \frac{\prod_{e \in E} \mbf{B}_N^{(\gamma(e))}(\phi(e))}{\prod_{e \in E} \sqrt{\xi_N^{(\gamma(e))}}} \\
    \leq &\prod_{[s]_{E, \pi} \in [r]/\sim_{E, \pi}} \sum_{\phi: \msr{V}([s]_{E, \pi}) \to [N]} \prod_{s' \in [s]_{E, \pi}} \left[ \left|\mbf{x}_N^{(s')}(\phi([v_{d_{s'}}^{(s')}]_\pi))\right| \left|\mbf{y}_N^{(s')}(\phi([v_0^{(s')}]_\pi))\right| \frac{\prod_{e \in E_{s'}} \mbf{B}_N^{(\gamma(e))}(\phi(e))}{\prod_{e \in E_{s'}} \sqrt{\xi_N^{(\gamma(e))}}} \right].
\end{align*}
Thus, it suffices to prove that
\begin{align*}
    &\sum_{\phi: \msr{V}([s]_{E, \pi}) \to [N]} \prod_{s' \in [s]_{E, \pi}} \left[ \left|\mbf{x}_N^{(s')}(\phi([v_{d_{s'}}^{(s')}]_\pi))\right| \left|\mbf{y}_N^{(s')}(\phi([v_0^{(s')}]_\pi))\right| \frac{\prod_{e \in E_{s'}} \mbf{B}_N^{(\gamma(e))}(\phi(e))}{\prod_{e \in E_{s'}} \sqrt{\xi_N^{(\gamma(e))}}} \right] \\
    = &O_{p_1, \ldots, p_r}\left(\left[\min_{i \in I_{p_1, \ldots, p_r}} \sqrt{\xi_N^{(i)}} \right]^{-\#([s]_{E, \pi})} \right).
\end{align*}
Without loss of generality, we may then assume that there is only one equivalence class $[s]_{E, \pi} = [r]$, which allows us to cut down on notation. In particular, we have reduced the problem to establishing
\begin{equation}\label{eq:deterministic_bound}
\begin{aligned}
    &\sum_{\phi: V^\pi \to [N]} \prod_{s = 1}^r \left[ \left|\mbf{x}_N^{(s)}(\phi([v_{d_s}^{(s)}]_\pi))\right| \left|\mbf{y}_N^{(s)}(\phi([v_0^{(s)}]_\pi))\right| \frac{\prod_{e \in E_s} \mbf{B}_N^{(\gamma(e))}(\phi(e))}{\prod_{e \in E_s} \sqrt{\xi_N^{(\gamma(e))}}} \right] \\
    = &O_{p_1, \ldots, p_r}\left(\left[ \min_{i \in I_{p_1, \ldots, p_r}} \sqrt{\xi_N^{(i)}} \right]^{-r} \right).
\end{aligned}
\end{equation}

Intuitively, each test graph $T_s$ is responsible for a factor of $\left[\min_{i \in I_{p_1, \ldots, p_r}} \sqrt{\xi_N^{(i)}}\right]^{-1}$ via the unit vectors $\mbf{x}_N^{(s)}, \mbf{y}_N^{(s)}$ or a defect in the underlying simple graph $\underline{G^\pi} = (V^\pi, [E]_\pi^{(2)})$ from an edge overlay. To formalize this, it will be convenient to introduce some additional notation to gather the relevant terms. We define
\begin{align*}
    (w_1, \ldots, w_{2r}) &:= (v_0^{(1)}, v_{d_1}^{(1)},  \ldots, v_0^{(r)}, v_{d_1}^{(r)}); \\
    (\mbf{u}_N^{(1)}, \ldots, \mbf{u}_N^{(2r)}) &:= (|\mbf{x}_N^{(1)}|, |\mbf{y}_N^{(1)}|, \ldots, |\mbf{x}_N^{(r)}|, |\mbf{y}_N^{(r)}|) \in (\sphr_{\geq 0}^{N-1})^{2r},
\end{align*}
in which case
\begin{align*}
    &\prod_{s = 1}^r \left[ \left|\mbf{x}_N^{(s)}(\phi([v_{d_s}^{(s)}]_\pi))\right| \left|\mbf{y}_N^{(s)}(\phi([v_0^{(s)}]_\pi))\right| \frac{\prod_{e \in E_s} \mbf{B}_N^{(\gamma(e))}(\phi(e))}{\prod_{e \in E_s} \sqrt{\xi_N^{(\gamma(e))}}} \right] \\
    = &\prod_{s = 1}^{2r} \mbf{u}_N^{(s)}(\phi([w_s]_\pi)) \frac{\prod_{[e] \in [E]_\pi} \mbf{B}_N^{(\gamma(\lfloor e \rfloor_\pi))}(\phi(\lfloor e \rfloor_\pi))}{\prod_{[e] \in [E]_\pi} \prod_{e' \in [e]_\pi} \sqrt{\xi_N^{(\gamma(e'))}}} \\
    \leq &\prod_{s = 1}^{2r} \mbf{u}_N^{(s)}(\phi([w_s]_\pi)) \frac{\prod_{[e] \in [F]_\pi} \mbf{B}_N^{(\gamma(\lfloor e \rfloor_\pi))}(\phi(\lfloor e \rfloor_\pi))}{\prod_{[e] \in [E]_\pi} \prod_{e' \in [e]_\pi} \sqrt{\xi_N^{(\gamma(e'))}}}
\end{align*}
for any spanning tree $(V^\pi, [F]_\pi)$ of $\underline{G^\pi}$. The unit vectors further reduce the numerator by virtue of the Cauchy-Schwarz inequality, which implies that for any $S \subset [N]$,
\begin{equation}\label{eq:cauchy_schwarz}
\begin{aligned}
    \sum_{i \in S} \mbf{u}_N^{(s)}(i) &\leq \sqrt{\#(S)}, \\
    \sum_{i \in S} \mbf{u}_N^{(s)}(i)\mbf{u}_N^{(s')}(i) &\leq 1;
\end{aligned}
\end{equation}
however, in general, one cannot do better than
\begin{equation}\label{eq:cauchy_schwarz_stack}
    \sum_{i \in S} \prod_{s \in \msr{S}} \mbf{u}_N^{(s)}(i) \leq 1
\end{equation}
for $\msr{S} \subset [2r]$ such that $\#(\msr{S}) \geq 2$, where we have used the fact that $\norm{\mbf{u}_N^{(s)}}_\infty \leq \norm{\mbf{u}_N^{(s)}}_2 = 1$.

To keep track of the unit vectors, we distinguish the vertices $V_{\op{out}} := \{w_1, \ldots, w_{2r}\} \subset V$ by calling them \emph{outer}. We refer to the remaining vertices $V_{\op{in}} := V \setminus V_{\op{out}}$ as \emph{inner}. For a partition $\pi \in \mcal{P}(V)$, we separate the blocks according to their inner/outer composition:
\[
    \pi^{(a, b)} = \{B \in \pi : \#(B \cap V_{\op{out}}) = a, \#(B \cap V_{\op{in}}) = b\}.
\]
We separate the vertices in a similar manner:
\[
    V_\pi^{(a, b)} = \bigcup_{B \in \pi^{(a, b)}} B.
\]
By a slight abuse of notation, we also write expressions such as $\pi^{(a, \geq b)}$ and $V_\pi^{(a, \geq b)}$ for the obvious analogues.

Without the unit vectors, we have the equality
\[
    \sum_{\phi: V^\pi \to [N]} \frac{\prod_{[e] \in [F]_\pi} \mbf{B}_N^{(\gamma(\lfloor e \rfloor_\pi))}(\phi(\lfloor e \rfloor_\pi))}{\prod_{[e] \in [E]_\pi} \prod_{e' \in [e]_\pi} \sqrt{\xi_N^{(\gamma(e'))}}} = \frac{N \prod_{[e] \in [F]_\pi} \xi_N^{(\gamma(\lfloor e \rfloor_\pi))}}{\prod_{[e] \in [E]_\pi} \prod_{e' \in [e]_\pi} \sqrt{\xi_N^{(\gamma(e'))}}},
\]
where $\#(\pi) = \#([F]_\pi) + 1$ since $(V^\pi, [F]_\pi)$ is a spanning tree. We know how to remove the factor of $N$ in the numerator at the cost of either two blocks in $\#(\pi^{(1 , \geq 0)})$ or one block in $\#(\pi^{(\geq 2, \geq 0)})$ from the proof of Proposition \ref{prop:isotropic_global_law}. When assigning the remaining values of $\phi$, we can use \eqref{eq:cauchy_schwarz} and \eqref{eq:cauchy_schwarz_stack} to reduce a term $\xi_N^{(\gamma(\lfloor e \rfloor_\pi))}$ in the numerator to either $\sqrt{\xi_N^{(\gamma(\lfloor e \rfloor_\pi))}}$ (if the block $B \in V^\pi$ belongs to $\pi^{(1, \geq 0)}$) or $1$ (if the block $B \in V^\pi$ belongs to $\pi^{(\geq 2, \geq 0)}$). Since $\pi \in \mcal{P}_+(V)$, we still have \eqref{eq:normalization_balance} to handle the blocks $B \in \pi^{(0, \geq 1)}$. Thus, reintroducing the unit vectors, we arrive at the bound 
\begin{align*}
    &\sum_{\phi: V^\pi \to [N]} \prod_{s = 1}^{2r} \mbf{u}_N^{(s)}(\phi([w_s]_\pi)) \frac{\prod_{[e] \in [F]_\pi} \mbf{B}_N^{(\gamma(\lfloor e \rfloor_\pi))}(\phi(\lfloor e \rfloor_\pi))}{\prod_{[e] \in [E]_\pi} \prod_{e' \in [e]_\pi} \sqrt{\xi_N^{(\gamma(e'))}}} \\
    \leq &\left[ \min_{i \in I_{p_1, \ldots, p_r}} \xi_N^{(i)} \right]^{\left(\left[\#(\pi) - 1\right] - \left[\frac{\#(\pi^{(1, \geq 0)})}{2} + \#(\pi^{(\geq 2, \geq 0)}) - 1 \right]\right) - \frac{\sum_{s = 1}^r d_s}{2}},
\end{align*}
where $d_s$ is the number of edges in the test graph $T_s$. Thus, \eqref{eq:deterministic_bound} will follow if we can prove that for $\pi \in \mcal{P}_+(V)$ such that $[s]_{E, \pi} = [r]$, 
\[
    \left[\#(\pi) -  \frac{\#(\pi^{(1, \geq 0)})}{2} - \#(\pi^{(\geq 2, \geq 0)})\right] - \frac{\sum_{s = 1}^r d_s}{2} \leq -\frac{r}{2}, 
\]
or, equivalently,
\begin{equation}\label{eq:combinatorial_concentration_bound}
    \frac{\#(\pi^{(1, \geq 1)})}{2} + \#(\pi^{(0, \geq 1)}) + \frac{r}{2} \leq \frac{\sum_{s = 1}^r d_s}{2},
\end{equation}
where we have used the fact that $\pi \in \mcal{P}_+(V)$ forces $\pi^{(1, \geq 0)} = \pi^{(1, \geq 1)}$.

We prove \eqref{eq:combinatorial_concentration_bound} by induction on the total number of edges $D = \sum_{s = 1}^r d_s$ with the base cases $D = 2, 3$ (cf.\@ \cite[Proposition 4.4]{AGV22}). Note that centrality allows us to restrict to $r \geq 2$ since otherwise the moment bound is trivially true.  The case of $D = 2$ then corresponds to $r = 2$ and $d_1 = d_2 = 1$, which has no inner vertices. Thus, $\pi^{(1, \geq 1)} = \pi^{(0, \geq 1)} = \emptyset$, and \eqref{eq:combinatorial_concentration_bound} follows.

If $D = 3$, then there are two possibilities. First, it could be that $r = 3$ and $d_1 = d_2 = d_3 = 1$. As before, there are no inner vertices in this case, and so we are done. If $r = 2$, then $\{d_1, d_2\} = \{1, 2\}$. Thus, there is exactly one inner vertex; however, the mandatory edge overlay between $T_1$ and $T_2$ means that this lone inner vertex will be merged with at least one outer vertex, whence $\#(\pi^{(0, \geq 1)}) = 0$ and $\#(\pi^{(1, \geq 1)}) \leq 1$. Plugging everything in, we again have \eqref{eq:combinatorial_concentration_bound}.

Now suppose that $D \geq 4$. If $\#(\pi^{(0, 1)}) \leq \#(\pi^{(1, \geq 2)})$, then we are done. Indeed, in this case, \small
\begin{align*}
    \frac{\#(\pi^{(1, \geq 1)})}{2} + \#(\pi^{(0, \geq 1)}) &= \frac{\#(\pi^{(1, 1)}) + \#(\pi^{(1, \geq 2)})}{2} + \#(\pi^{(0,  1)}) + \#(\pi^{(0, \geq 2)}) \\
    &\leq \frac{\#(\pi^{(1, 1)})}{2} + \#(\pi^{(1, \geq 2)}) + \frac{\#(\pi^{(0,  1)})}{2} + \#(\pi^{(0, \geq 2)}) \\
    &\leq \frac{\#(V_\pi^{(1, 1)} \cap V_{\op{in}})}{2} + \frac{\#(V_\pi^{(1, \geq 2)} \cap V_{\op{in}})}{2} + \frac{\#(V_\pi^{(0,  1)} \cap V_{\op{in}})}{2} + \frac{\#(V_\pi^{(0, \geq 2)} \cap V_{\op{in}})}{2} \\
    &\leq \frac{\#(V_{\op{in}})}{2} \\
    &= \frac{\sum_{s = 1}^r (d_s - 1)}{2} \\
    &= \frac{D - r}{2}.
\end{align*} \normalsize
So, we can assume that $\#(\pi^{(0, 1)}) > \#(\pi^{(1, \geq 2)})$. Let $B \in \pi^{(0, 1)} \neq \emptyset$. This means that $B = \{v_0\}$ consists of a single inner vertex, say belonging to the test graph $T_{s_0}$. Since $\pi \in \mcal{P}_+(V)$, it must be that the two vertices $v_1, v_2$ adjacent to $v_0$ in $T_{s_0}$ are identified by $\pi$. This corresponds to pinching off the two edges $e_1, e_2$ incident to $v_0$ in $T_{s_0}$. The block $B$ is then necessarily a leaf in $\underline{T^\pi}$ with exactly two incident edges in $T^\pi$. The remainder of $T^\pi$ can therefore be constructed as a quotient of the disjoint union of $(T_s)_{s \in [r]\setminus\{s_0\}}$ and a shortened version of $T_{s_0}$ by two edges. See Figure \ref{fig:induction} for an illustration. Note that we must have $d_{s_0} \geq 3$; otherwise, $d_{s_0} = 2$ and an edge overlay between $T_{s_0}$ and any other $T_s$ would contradict $B \in \pi^{(0, 1)}$.

\begin{figure}
      \begin{tikzpicture}[shorten > = 2.5pt, baseline=(current  bounding  box.center)]

    \node at (2, -1.125) {$T_{s_0}$};
    
    \draw[fill=black] (0, 0) circle (1.75pt);
    \draw[fill=black] (1,0) circle (1.75pt);
    \draw[fill=black] (2, 0) circle (1.75pt);
    \node at (2, -.375) {\footnotesize$v_1$\normalsize};
    \draw[fill=black] (3, 0) circle (1.75pt);
    \node at (3, -.375) {\footnotesize$v_0$\normalsize};
    \draw[fill=black] (4, 0) circle (1.75pt);
    \node at (4, -.375) {\footnotesize$v_2$\normalsize};

    \draw[semithick, ->] (1, 0) to (0, 0);
    \draw[semithick, ->] (2, 0) to (1, 0);
    \draw[semithick, ->] (3, 0) to node[pos=.375, above] {\footnotesize$e_1$\normalsize} (2, 0);
    \draw[semithick, ->] (4, 0) to node[pos=.375, above] {\footnotesize$e_2$\normalsize} (3, 0);

    \draw[dashed] (4,0) arc (0:180:1) node[midway, above]{\footnotesize$\pi$\normalsize};
    
    \draw[fill=black] (7, 0) circle (1.75pt);
    \draw[fill=black] (8, 0) circle (1.75pt);
    \draw[fill=black] (9, 0) circle (1.75pt);
    \node at (9, -.375) {\footnotesize$\{v_1, v_2\}$\normalsize};
    \draw[fill=black] (9, 1) circle (1.75pt);
    \node at (9, 1.5) {\footnotesize$B$\normalsize};

    \draw[semithick, ->] (8, 0) to (7, 0);
    \draw[semithick, ->] (9, 0) to (8, 0);
    \draw[semithick, ->] (8.9375, 1) to (8.9375, 0);
    \draw[semithick, ->] (9.0625, 0) to (9.0625, 1);

    \node at (13, -1.125) {$\wtilde{T}_{s_0}$};
    
    \draw[fill=black] (12, 0) circle (1.75pt);
    \draw[fill=black] (13, 0) circle (1.75pt);
    \draw[fill=black] (14, 0) circle (1.75pt);
    \node at (14, -.375) {\footnotesize$\tilde{v}_2$\normalsize};

    \draw[semithick, ->] (13, 0) to (12, 0);
    \draw[semithick, ->] (14, 0) to (13, 0);    
  \end{tikzpicture}
\caption{An example of the pinching off of the two edges adjacent to $v_0$ in $T_{s_0}$ as necessitated by $\pi \in \mcal{P}_+(V)$. Since $B = \{v_0\} \in \pi^{(0, 1)}$, this is the only way the edges adjacent to $v_0$ can satisfy the condition $\#([e]_\pi) \geq 2$.}\label{fig:induction}
\end{figure}
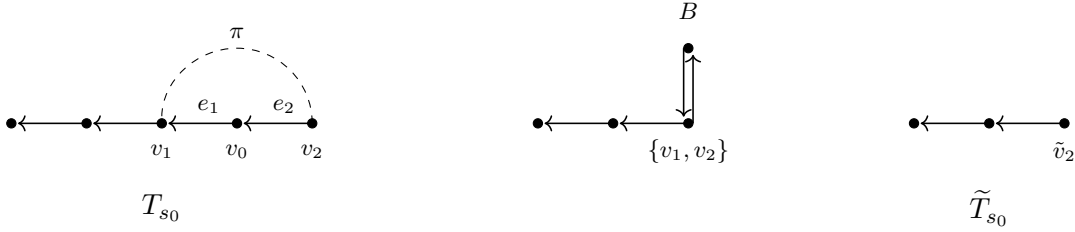

To apply the induction hypothesis, let $\wtilde{T}_{s_0}$ be the test graph obtained from $T_{s_0}$ by removing $v_0$, its two incident edges, and merging $v_1$ and $v_2$ into a vertex $\tilde{v}_2$. If $v_2$ is inner (resp., outer), then so too is $\tilde{v}_2$ in $\wtilde{T}_{s_0}$. The partition $\pi$ defines a natural partition $\tilde{\pi}$ of the vertices of the disjoint union $\sqcup_{s \in [r]\setminus\{s_0\}} T_s \sqcup \wtilde{T}_{s_0}$ as follows. For a block $B' \in \pi\setminus\{B\}$, we define
\begin{equation}\label{eq:projected_block}
    f(B') = \begin{dcases}
    (B'\setminus\{v_1, v_2\})\cup\{\tilde{v}_2\} &\text{if } v_1, v_2 \in B'; \\
    B' &\text{else}.
    \end{dcases}
\end{equation}
These new blocks make up the partition
\[
    \tilde{\pi} = \{f(B') : B' \in \pi\setminus\{B\}\} \in \mcal{P}\Big((V\setminus\{v_0, v_1, v_2\}) \cup \{\tilde{v}_2\}\Big).
\]
Since $B$ was a leaf in $\underline{T^\pi}$ with exactly two incident edges $e_1, e_2$ in $T^\pi$, the partition $\tilde{\pi}$ still satisfies
\begin{equation}\label{eq:stacked_edges}
    \#([e]_{\tilde{\pi}}) \geq 2, \qquad \forall e \in E\setminus\{e_1, e_2\},
\end{equation}
meaning $\tilde{\pi} \in \mcal{P}_+\Big((V\setminus\{v_0, v_1, v_2\}) \cup \{\tilde{v}_2\}\Big)$. Furthermore, as noted earlier, neither $e_1$ nor $e_2$ can participate in an edge overlay between $T_{s_0}$ and some other $T_s$ due to the fact that $B \in \pi^{(1, 0)}$. So, any such overlay is preserved in $(\sqcup_{s \in [r]\setminus\{s_0\}} T_s \sqcup \wtilde{T}_{s_0})^{\tilde{\pi}}$, meaning there is still only one equivalence class $[s]_{E\setminus\{e_1, e_2\}, \tilde{\pi}} = [r]$. The induction hypothesis then allows us to conclude that
\[
    \frac{\#(\tilde{\pi}^{(1, \geq 1)})}{2} + \#(\tilde{\pi}^{(0, \geq 1)}) + \frac{r}{2} \leq \frac{\sum_{s \in [r]\setminus{s_0}} d_s + \tilde{d}_{s_0}}{2} = \frac{\sum_{s = 1}^r d_s - 2}{2}.
\]

We must now relate $\tilde{\pi}^{(1, \geq 1)}$ and $\tilde{\pi}^{(0, \geq 1)}$ to $\pi^{(1, \geq 1)}$ and $\pi^{(0, \geq 1)}$ respectively. By definition \eqref{eq:projected_block}, $f$ changes the composition of exactly one block $[v_1]_\pi = [v_2]_\pi$, decreasing the number of inner vertices in this block by one and leaving all other blocks untouched. If $v_2$ is outer, then $[v_2]_\pi \in \pi^{(\geq 1, \geq 1)}$ and $f([v_2]_\pi) \in \tilde{\pi}^{(\geq 1, \geq 0)}$. If $v_2$ is inner, then $[v_2]_\pi \in \pi^{(\geq 0, \geq 2)}$ and $f([v_2]_\pi) \in \tilde{\pi}^{(\geq 0, \geq 1)}$. In either case, the map $f$ restricts to a bijection between $\pi^{(0, \geq 1)}\setminus\{B\}$ and $\tilde{\pi}^{(0, \geq 1)}$, whence
\[
    \#(\tilde{\pi}^{(0, \geq 1)}) = \#(\pi^{(0, \geq 1)}) - 1.
\]
Similarly, $f$ restricts to a bijection between $\pi^{(1, \geq 0)}$ and $\tilde{\pi}^{(1, \geq 0)}$; however, as we have already seen, condition \eqref{eq:restricted_partitions} (resp., \eqref{eq:stacked_edges}) forces $\pi^{(1, \geq 0)} = \pi^{(1, \geq 1)}$ (resp., $\tilde{\pi}^{(1, \geq 0)} = \tilde{\pi}^{(1, \geq 1)}$), whence
\[
    \#(\pi^{(1, \geq 1)}) = \#(\tilde{\pi}^{(1, \geq 1)}). 
\]
Putting everything together, we obtain
\begin{align*}
    \frac{\#(\pi^{(1, \geq 1)})}{2} + \#(\pi^{(0, \geq 1)}) + \frac{r}{2} &= \#(\tilde{\pi}^{(1, \geq 1)}) + \#(\tilde{\pi}^{(0, \geq 1)}) + 1 + \frac{r}{2}\\
    &\leq \frac{\sum_{s = 1}^r d_s - 2}{2} + 1 \\
    &= \frac{\sum_{s = 1}^r d_s}{2},
\end{align*}
as was to be shown. 
\end{proof}

\begin{rem}\label{rem:central_moments_k_regular_wigner}
In the case of independent $(k_N^{(i)})_{i \in I}$-regular Wigner matrices, one needs only to carry forward the same modifications from Remark \ref{rem:expectation_k_regular_wigner}. The upper bound in Lemma \ref{lem:central_moments} is easily seen to be sharp and can be achieved by overlaying copies of lines to obtain a forest of double trees.
\end{rem}

\begin{cor}[Concentration]\label{cor:concentration}
For $p(\vec{z}) \in \C\langle z_i : i \in I \rangle$, $r \in \N$, and $\mbf{x}_N, \mbf{y}_N \in \sphr^{N-1}$,
\[
    \pr\left(\left|\Tr(p(\mcal{Z}_N) \mbf{x}_N \mbf{y}_N^*) - \E\left[\Tr(p(\mcal{Z}_N) \mbf{x}_N \mbf{y}_N^*)\right]\right| \geq \varepsilon \right) = O_{p, r, \varepsilon}\left(\left[\min_{i \in I_p} \sqrt{\xi_N^{(i)}}\right]^{-r}\right).
\]
\end{cor}
\begin{proof}
Since $b_N^{(i)} \gg 1$ for each $i \in I$, it suffices to prove the result for $r = 2m$ even. We define a conjugate linear involution $*: \C\langle z_i : i \in I \rangle \to \C\langle z_i : i \in I \rangle$ by its action on monomials:
\[
    (z_{i(1)} \cdots z_{i(d)})^* = z_{i(d)} \cdots z_{i(1)}.
\]
Since the matrices $\mcal{Z}_N = (\mbf{\Xi}_N^{(i)})_{i \in I}$ are Hermitian, this operation commutes with the usual matrix adjoint:
\[
    [p(\mcal{Z}_N)]^* = p^*(\mcal{Z}_N),
\]
We use this to write the complex conjugate of our weighted trace as yet another weighted trace
\[
    \overline{\Tr(p(\mcal{Z}_N) \mbf{x}_N \mbf{y}_N^*)} = \Tr\left(\left[p(\mcal{Z}_N) \mbf{x}_N \mbf{y}_N^*\right]^*\right) = \Tr(p^*(\mcal{Z}_N) \mbf{y}_N \mbf{x}_N^*).
\]
In particular, the squared modulus can be written as a product
\begin{align*}
    &\left|\Tr(p(\mcal{Z}_N) \mbf{x}_N \mbf{y}_N^*) - \E\left[\Tr(p(\mcal{Z}_N) \mbf{x}_N \mbf{y}_N^*)\right]\right|^2 \\
    = &\Big(\Tr(p(\mcal{Z}_N) \mbf{x}_N \mbf{y}_N^*) - \E\left[\Tr(p(\mcal{Z}_N) \mbf{x}_N \mbf{y}_N^*)\right]\Big)\Big(\Tr(p^*(\mcal{Z}_N) \mbf{y}_N \mbf{x}_N^*) - \E\left[\Tr(p^*(\mcal{Z}_N) \mbf{y}_N \mbf{x}_N^*)\right]\Big).
\end{align*}
The result then follows from Lemma \ref{lem:central_moments} and Markov's inequality.
\end{proof}

We can now prove the isotropic global law. We recall the notation $\pto$ for convergence in probability.

\begin{prop}[Isotropic global law]\label{prop:isotropic_global_law}
Let $(\mbf{x}_N^{(s)})_{s = 1}^r, (\mbf{y}_N^{(s)})_{s = 1}^r \subset \sphr^{N-1}$ be such that
\[
  \lim_{N \to \infty} \inn{\mbf{x}_N^{(s)}}{\mbf{y}_N^{(s)}} = c_s.
\]
If $p_1(\vec{z}), \ldots, p_r(\vec{z}) \in \C\langle z_i: i \in I\rangle$, then  
\[
  \Tr\left(\prod_{s = 1}^r \mbf{x}_{N}^{(s-1)}{\mbf{y}_{N}^{(s)}}^* p_s(\mcal{Z}_N)\right) \pto \prod_{s = 1}^r \big[c_{s}\tau_{\mathcal{Z}}(p_s)\big],
\]
where $\mbf{x}_N^{(r)} = \mbf{x}_N^{(0)}$. If $\min_{i \in I_{p_1, \ldots, p_r}} \xi_N^{(i)} \gg N^\varepsilon$ for some $\varepsilon > 0$, then this convergence can be upgraded to the almost sure sense:
\[
  \lim_{N \to \infty} \Tr\left(\prod_{s = 1}^r \mbf{x}_{N}^{(s-1)}{\mbf{y}_{N}^{(s)}}^* p_s(\mcal{Z}_N)\right) \aseq \prod_{s = 1}^r \big[c_{s}\tau_{\mathcal{Z}}(p_s)\big].    
\]
\end{prop}
\begin{proof}
We start by rewriting the trace in question into a product of traces:
\[
    \Tr\left(\prod_{s = 1}^r \mbf{x}_{N}^{(s-1)}{\mbf{y}_{N}^{(s)}}^* p_s(\mcal{Z}_N)\right) = \prod_{s = 1}^r \left\langle p_s(\mcal{Z}_N)\mbf{x}_{N}^{(s)}, {\mbf{y}_{N}^{(s)}} \right\rangle = \prod_{s = 1}^r \Tr\left(p_s(\mcal{Z}_N) \mbf{x}_N^{(s)} {\mbf{y}_N^{(s)}}^*\right).
\]
Thus, it suffices to prove the stated convergence for a single term $\Tr\left(p_s(\mcal{Z}_N) \mbf{x}_N^{(s)} {\mbf{y}_N^{(s)}}^*\right)$. Convergence in probability follows from Lemma \ref{lem:isotropic_global_law_expectation} and Corollary \ref{cor:concentration}. If $\min_{i \in I_{p_1, \ldots, p_r}} \xi_N^{(i)} \gg N^\varepsilon$ for some $\varepsilon > 0$, then we can choose a sufficiently large value of $r$ in Corollary \ref{cor:concentration} to apply the Borel-Cantelli lemma and upgrade the convergence to the almost sure sense.
\end{proof}

\subsection{Proof of Theorem \ref{thm:spiked_rbm}}\label{sec:spiked_rbm}
We elaborate on the outline of the proof given in the introduction. While many of the details are routine, we commit them here for completeness. Recall the notation $\mbf{M}_N = \mbf{\Xi}_N + \mbf{A}_N = \sum_{k = 1}^N \lambda_k(\mbf{M}_N)\mbf{m}_N^{(k)}{\mbf{m}_N^{(k)}}^*$ for the spectral decomposition of the spiked RBM model, and likewise for the perturbation $\mbf{A}_N = \sum_{s = 1}^r \theta_s \mbf{a}_N^{(s)}{\mbf{a}_N^{(s)}}^*$. We assume that $\xi_N \gg N^{\varepsilon}$ for some $\varepsilon > 0$. We start with a straightforward consequence of the isotropic global law. 

\begin{lemma}\label{lem:spectral_measure}
For $s' \in [r]$, the spectral measure $\mu_{\mbf{M}_N}^{\mbf{a}_N^{(s')}}$ converges weakly almost surely to
\[
    \mu_{\theta_{s'}}(dx) = \frac{\indc{|x| \leq 2\sigma}}{2\pi}\frac{\sqrt{4\sigma^2 - x^2}}{\theta_{s'}^2 + \sigma^2 - \theta_{s'} x} \, dx + \indc{|\theta_{s'}| > \sigma} \left(1-\frac{\sigma^2}{\theta_{s'}^2}\right)\delta_{\theta_{s'} + \frac{\sigma^2}{\theta_{s'}}}(dx).
\]
\end{lemma}
\begin{proof}
Since the eigenvectors $(\mbf{a}_N^{(s)})_{s \in [r]}$ are orthonormal, Proposition \ref{prop:isotropic_global_law} tells us that 
\begin{align*}
    \lim_{N \to \infty} \left\langle \mbf{M}_N^m \mbf{a}_N^{(s')}, \mbf{a}_N^{(s')} \right\rangle &= \lim_{N \to \infty} \Tr\left(\mbf{a}_N^{(s')} {\mbf{a}_N^{(s')}}^* \left(\mbf{\Xi}_N + \sum_{s = 1}^r \theta_s \mbf{a}_N^{(s)}{\mbf{a}_N^{(s)}}^*\right)^m\right) \\
    &\aseq \lim_{N \to \infty} \Tr\left(\mbf{a}_N^{(s')} {\mbf{a}_N^{(s')}}^* \left(\mbf{W}_N + \theta_{s'} \mbf{a}_N^{(s')}{\mbf{a}_N^{(s')}}^*\right)^m\right) \\
    &= \lim_{N \to \infty} \left\langle \left(\mbf{W}_N + \theta_{s'} \mbf{a}_N^{(s')}{\mbf{a}_N^{(s')}}^*\right)^m \mbf{a}_N^{(s')}, \mbf{a}_N^{(s')} \right\rangle,
\end{align*}
where $\mbf{W}_N$ is a Wigner matrix. Thus, the moments of $\mu_{\mbf{M}_N}^{\mbf{a}_N^{(s')}}$ and the moments of $\mu_{\mbf{W}_N + \theta_{s'} \mbf{a}_N^{(s')}{\mbf{a}_N^{(s')}}^*}^{\mbf{a}_N^{(s')}}$ converge to the same deterministic sequence $(m_1, m_2, \ldots) \in \R^\N$ almost surely. Noiry proved that the spectral measure $\mu_{\mbf{W}_N + \theta_{s'} \mbf{a}_N^{(s')}{\mbf{a}_N^{(s')}}^*}^{\mbf{a}_N^{(s')}}$ converges weakly almost surely to $\mu_{\theta_{s'}}$ \cite[Proposition 2]{Noi21}: the finiteness of the limiting moments $m_i < \infty$ further implies that the moments of $\mu_{\theta_{s'}}$ are given by the same sequence $(m_1, m_2, \ldots)$. Being compactly supported, the distribution $\mu_{\theta_{s'}}$ is uniquely determined by its moments. Consequently, the moment convergence of $\mu_{\mbf{M}_N}^{\mbf{a}_N^{(s')}}$ to $(m_1, m_2, \ldots)$ implies that $\mu_{\mbf{M}_N}^{\mbf{a}_N^{(s')}}$ converges weakly almost surely to $\mu_{\theta_{s'}}$.
\end{proof}

To prove the eigenvalue BBP transition \ref{spiked_edge} for $\mbf{M}_N$, we use the classical Weyl interlacing inequality specialized to a rank one perturbation \cite[Corollary 4.3.9]{HJ13}.

\begin{prop}[Weyl]\label{prop:interlacing}
Let $\mbf{H}_N \in \matn_N(\C)$ be Hermitian and $\mbf{v}_N \in \sphr^{N-1}$. If $\theta > 0$, then 
\begin{align*}
    \lambda_k(\mbf{H}_N) &\leq \lambda_k(\mbf{H}_N + \theta \mbf{v}_N\mbf{v}_N^*) \leq \lambda_{k+1}(\mbf{H}_N), \qquad k = 1, \ldots, N-1; \\
    \lambda_N(\mbf{H}_N) &\leq \lambda_N(\mbf{H}_N + \theta \mbf{v}_N\mbf{v}_N^*).
\end{align*}
\end{prop}

We can now give the 

\begin{proof}[Proof of Theorem \ref{thm:spiked_rbm}]
First, assume that the $(\theta_s)_{s = 1}^r$ are distinct. We prove the result by induction on $r$. The base case of $r = 1$ corresponds to a rank one perturbation $\mbf{M}_N = \mbf{\Xi}_N + \theta_1\mbf{a}_N^{(1)}{\mbf{a}_N^{(1)}}^*$. Without loss of generality, we may assume that $\theta_1 > 0$. Applying Lemma \ref{lem:spectral_measure} to the explicit form of the spectral measure in Example \ref{eg:spectral_measure_vector}, we obtain the almost sure weak convergence
\begin{equation}\label{eq:BBP_rank_one}
\begin{aligned}
    \lim_{N \to \infty} \mu_{\mbf{M}_N}^{\mbf{a}_N^{(1)}} &= \lim_{N \to \infty} \sum_{k = 1}^N \left|\left\langle \mbf{a}_N^{(1)}, \mbf{m}_N^{(k)} \right\rangle\right|^2 \delta_{\lambda_k(\mbf{M}_N)} \\
    &\aseq \frac{\indc{|x| \leq 2\sigma}}{2\pi}\frac{\sqrt{4\sigma^2 - x^2}}{\theta_1^2 + \sigma^2 - \theta_1 x} \, dx + \indc{|\theta_1| > \sigma} \left(1-\frac{\sigma^2}{\theta_1^2}\right)\delta_{\theta_1 + \frac{\sigma^2}{\theta_1}}(dx).
\end{aligned}
\end{equation}
The strong convergence of $\mbf{\Xi}_N$ and the interlacing inequality imply that there is at most one outlier:
\begin{align*}
    \lim_{N \to \infty} \lambda_1(\mbf{M}_N) &\aseq -2\sigma;\\
    \lim_{N \to \infty} \lambda_{N-1}(\mbf{M}_N) &\aseq 2\sigma.
\end{align*}
In particular, if $\theta_1 \in (\sigma, \infty)$, then the atom located at $\theta_1 + \frac{\sigma^2}{\theta_1} > 2\sigma$ in the limiting spectral measure \eqref{eq:BBP_rank_one} must originate from $\lambda_N(\mbf{M}_N)$, whence
\begin{align*}
    \lim_{N \to \infty} \lambda_N(\mbf{M}_N) &\aseq \theta_1 + \frac{\sigma^2}{\theta_1}; \\
    \lim_{N \to \infty} \left|\left\langle \mbf{a}_N^{(1)}, \mbf{m}_N^{(N)} \right\rangle\right|^2 &\aseq 1-\frac{\sigma^2}{\theta_1^2}.
\end{align*}
Note that the interlacing inequality also implies that $\lambda_N(\mbf{M}_N)$ is a nondecreasing function of $\theta_1 > 0$. If $\theta_1 \in (0, \sigma]$, then we can use this monotonicity to deduce that
\[
    \lim_{N\to \infty} \lambda_N(\mbf{M}_N) \aseq 2\sigma.
\]
The lack of an atom at $2\sigma$ in the limiting spectral measure \eqref{eq:BBP_rank_one} in this case then implies
\[  
    \lim_{N \to \infty} \left|\left\langle \mbf{a}_N^{(1)}, \mbf{m}_N^{(N)} \right\rangle\right|^2 \aseq 0.
\]

Now assume the result for some $r \geq 1$ and consider a rank $r+1$ perturbation
\begin{align*}
    \mbf{M}_N &= \mbf{\Xi}_N + \sum_{s = 1}^{r+1} \theta_s\mbf{a}_N^{(s)}{\mbf{a}_N^{(s)}}^* \\
    &= \left(\mbf{\Xi}_N + \sum_{s = 1}^{r} \theta_s\mbf{a}_N^{(s)}{\mbf{a}_N^{(s)}}^*\right) + \theta_{r+1}\mbf{a}_N^{(r+1)}{\mbf{a}_N^{(r+1)}}^* \\
    &= \wtilde{\mbf{M}}_N + \theta_{r+1}\mbf{a}_N^{(r+1)}{\mbf{a}_N^{(r+1)}}^*.
\end{align*}
We recall the assumption $\theta_1 < \cdots < \theta_{r+1}$ are nonzero and the notation
\begin{align*}
    L_{-\sigma} &= \#(\{s \in [r]: \theta_s < -\sigma\}); \\
    L_{+\sigma} &= \#(\{s \in [r]: \theta_s > \sigma\}).
\end{align*}    
Without loss of generality, we may assume that $\theta_{r+1} > 0$. By the induction hypothesis, we know that $\wtilde{\mbf{M}}_N = \sum_{k = 1}^N \lambda_k(\wtilde{\mbf{M}}_N) \wtilde{\mbf{m}}_N^{(k)} \wtilde{\mbf{m}}_N^{(k)}{}^*$ satisfies
\begin{equation}\label{eq:BBP_induction_hypothesis}
\begin{aligned}
    \lim_{N \to \infty} \lambda_k(\wtilde{\mbf{M}}_N) &\aseq \theta_k + \frac{\sigma^2}{\theta_k} < -2\sigma, \qquad \forall k \in [L_{-\sigma}];\\
    \lim_{N \to \infty} \lambda_{L_{-\sigma} + 1}(\wtilde{\mbf{M}}_N) &\aseq -2\sigma; \\
    \lim_{N \to \infty} \lambda_{N + 1 - k}(\wtilde{\mbf{M}}_N) &\aseq \theta_{r+1-k} + \frac{\sigma^2}{\theta_{r+1-k}} > 2 \sigma, \qquad \forall k \in [L_{+\sigma}]; \\
    \lim_{N \to \infty} \lambda_{N - L_{+\sigma}}(\wtilde{\mbf{M}}_N) &\aseq 2 \sigma.
\end{aligned}
\end{equation}
Once again, we use Lemma \ref{lem:spectral_measure} to compute the limiting spectral measure
\begin{equation}\label{eq:BBP_spectral_measure}
\begin{aligned}
    \lim_{N \to \infty} \mu_{\mbf{M}_N}^{\mbf{a}_N^{(s)}} &= \lim_{N \to \infty} \sum_{k = 1}^N \left|\left\langle \mbf{a}_N^{(s)}, \mbf{m}_N^{(k)} \right\rangle\right|^2 \delta_{\lambda_k(\mbf{M}_N)} \\
    &\aseq \frac{\indc{|x| \leq 2\sigma}}{2\pi}\frac{\sqrt{4\sigma^2 - x^2}}{\theta_s^2 + \sigma^2 - \theta_s x} \, dx + \indc{|\theta_s| > \sigma} \left(1-\frac{\sigma^2}{\theta_s^2}\right)\delta_{\theta_s + \frac{\sigma^2}{\theta_s}}(dx)
\end{aligned}
\end{equation}
for any $s \in [r+1]$. Since the $(\theta_s)_{s=1}^{r+1}$ are distinct and the function $x \mapsto x + \frac{\sigma^2}{x}$ is injective for $|x| \geq \sigma$, the interlacing inequality applied to the rank one perturbation $\mbf{M}_N = \wtilde{\mbf{M}}_N + \theta_{r+1}\mbf{a}_N^{(r+1)}{\mbf{a}_N^{(r+1)}}^*$ and the convergences in \eqref{eq:BBP_induction_hypothesis} imply that the weak convergence in \eqref{eq:BBP_spectral_measure} holds only if
\begin{align*}
    \lim_{N \to \infty} \lambda_k(\mbf{M}_N) &\aseq \theta_k + \frac{\sigma^2}{\theta_k}, \qquad \forall k \in [L_{-\sigma}]; \\
    \lim_{N \to \infty} \left|\left\langle \mbf{a}_N^{(s)}, \mbf{m}_N^{(k)} \right\rangle\right|^2 &\aseq \indc{s = k}\left(1-\frac{\sigma^2}{\theta_k^2}\right), \qquad \forall (s, k) \in  [r+1] \times [L_{-\sigma}]; \\
    \lim_{N \to \infty} \lambda_{L_{-\sigma} + 1}(\mbf{M}_N) &\aseq -2\sigma; \\
    \lim_{N \to \infty} \left|\left\langle \mbf{a}_N^{(s)}, \mbf{m}_N^{(k)} \right\rangle\right|^2 &\aseq 0, \qquad \forall s \in [r+1] \text{ if } \theta_{k} \in [-\sigma, 0);\\
    \lim_{N \to \infty} \lambda_{N}(\mbf{M}_N) &\aseq \begin{dcases}
        \theta_{r+1} + \frac{\sigma^2}{\theta_{r+1}} &\text{if } \theta_{r+1} \in (\sigma, \infty); \\ 
        2 \sigma &\text{if } \theta_{r+1} \in (0, \sigma];
    \end{dcases} \\
    \lim_{N \to \infty} \left|\left\langle \mbf{a}_N^{(s)}, \mbf{m}_N^{(N)} \right\rangle\right|^2 &\aseq \begin{dcases} \indc{s = r+1}\left(1-\frac{\sigma^2}{\theta_{r+1}^2}\right) &\forall s \in [r+1] \text{ if } \theta_{r+1} \in (\sigma, \infty); \\ 0, &\forall s \in [r+1] \text{ if } \theta_{r+1} \in (0, \sigma].
    \end{dcases}
\end{align*}
Roughly speaking, we work our way in from the left edge of the spectrum using the trap $\lambda_1(\wtilde{\mbf{M}}_N) \leq \lambda_1(\mbf{M}_N)$. Having established the right edge of the spectrum using the lower bound $\lambda_N(\wtilde{\mbf{M}}_N) \leq \lambda_N(\mbf{M}_N)$, we can repeat the argument above and work our way in from the other direction using the trap $\lambda_{N-1}(\mbf{M}_N) \leq \lambda_N(\wtilde{\mbf{M}}_N)$. Thus,
\begin{align*}
    \lim_{N \to \infty} \lambda_{N-k}(\mbf{M}_N) &\aseq \theta_{r+1-k} + \frac{\sigma^2}{\theta_{r+1-k}}, \qquad \forall k \in [L_{+\sigma}]; \\
    \lim_{N \to \infty} \left|\left\langle \mbf{a}_N^{(s)}, \mbf{m}_N^{(N - k)} \right\rangle\right|^2 &\aseq \indc{s = r + 1 - k}\left(1-\frac{\sigma^2}{\theta_{r+1-k}^2}\right), \qquad \forall (s, k) \in  [r+1] \times [L_{+\sigma}]; \\
    \lim_{N \to \infty} \lambda_{N - 1 - L_{+\sigma}}(\mbf{M}_N) &\aseq 2\sigma; \\
    \lim_{N \to \infty} \left|\left\langle \mbf{a}_N^{(s)}, \mbf{m}_N^{(N - k)} \right\rangle\right|^2 &\aseq 0, \qquad \forall s \in [r+1] \text{ if } \theta_{r+1-k} \in (0, \sigma],
\end{align*}
which completes the induction step.

To prove the general case, we no longer assume that the nontrivial eigenvalues $\theta_1 \leq \cdots \leq \theta_r$ of the perturbation $\mbf{A}_N = \sum_{s = 1}^r \theta_s \mbf{a}_N^{(s)}{\mbf{a}_N^{(s)}}^*$ are necessarily simple. So, let $\Theta_1 < \cdots < \Theta_q$ be the distinct values of $(\theta_s)_{s=1}^r$ and $m_t = \dim(\ker(\Theta_t \mbf{I}_N - \mbf{A}_N))$ the multiplicity of $\Theta_t$. A standard continuity argument using the Hoffman-Wielandt inequality \cite[Corollary 6.3.8]{HJ13} proves the eigenvalue BBP transition \ref{spiked_edge} for $\mbf{M}_N$ from our earlier result in the case of distinct $(\theta_s)_{s = 1}^r$ (see, for example, \cite[Section 6.2.3]{BGN11}). From there, we can once again use the convergence of the spectral measure \eqref{eq:BBP_spectral_measure} to deduce that for any $(s, t) \in [r] \times [q]$,
\begin{align*}
    \lim_{N \to \infty} \sum_{k=\sum_{i = 1}^{t-1} m_i + 1}^{\sum_{i = 1}^{t} m_i} \left|\left\langle \mbf{a}_N^{(s)}, \mbf{m}_N^{(k)} \right\rangle\right|^2 &\aseq \indc{s \in \left[\sum_{i = 1}^{t-1} m_i + 1, \sum_{i = 1}^{t} m_i\right]}\left(1-\frac{\sigma^2}{\Theta_t^2}\right) \text{ if } \Theta_t < -\sigma; \\
    \lim_{N \to \infty} \sum_{k=\sum_{i = 1}^{t-1} m_i + 1}^{\sum_{i = 1}^{t} m_i} \left|\left\langle \mbf{a}_N^{(s)}, \mbf{m}_N^{(k)} \right\rangle\right|^2 &\aseq 0 \text{ if } \Theta_t \in [-\sigma, 0); \\
    \lim_{N \to \infty} \sum_{k = \sum_{i = t + 1}^q m_i + 1}^{\sum_{i = t}^{q} m_i} \left|\left\langle \mbf{a}_N^{(s)}, \mbf{m}_N^{(N+1-k)} \right\rangle\right|^2 &\aseq \indc{s \in \left[\sum_{i = 1}^{t-1} m_i + 1, \sum_{i = 1}^{t} m_i\right]}\left(1-\frac{\sigma^2}{\Theta_t^2}\right) \text{ if } \Theta_t > \sigma; \\
    \lim_{N \to \infty} \sum_{k = \sum_{i = t + 1}^q m_i + 1}^{\sum_{i = t}^{q} m_i} \left|\left\langle \mbf{a}_N^{(s)}, \mbf{m}_N^{(N+1-k)} \right\rangle\right|^2 &\aseq 0 \text{ if } \Theta_t \in (0, \sigma].
\end{align*}
Note that this already proves the nonalignment in the eigenvector BBP transition \ref{spiked_eigenvectors_II} and the second part of \ref{spiked_eigenvectors_I}; however, we do not have access to the projections of the individual eigenvectors $\mbf{m}_N^{(k)}/\mbf{m}_N^{(N+1-k)}$ onto the $\mbf{a}_N^{(s)}$. Summing the alignment over $s \in [r]$, we obtain the weaker statement
\begin{align*}
    \lim_{N \to \infty} \sum_{k=\sum_{i = 1}^{t-1} m_i + 1}^{\sum_{i = 1}^{t} m_i} \snorm{P_{\ker(\Theta_t \mbf{I}_N - \mbf{A}_N)}(\mbf{m}_N^{(k)})}_2^2 &\aseq m_t\left(1-\frac{\sigma^2}{\Theta_t^2}\right) \text{ if } \Theta_t < -\sigma; \\
    \lim_{N \to \infty} \sum_{k = \sum_{i = t + 1}^q m_i + 1}^{\sum_{i = t}^{q} m_i} \snorm{P_{\ker(\Theta_t \mbf{I}_N - \mbf{A}_N)}(\mbf{m}_N^{(N+1-k)})}_2^2 &\aseq m_t\left(1-\frac{\sigma^2}{\Theta_t^2}\right) \text{ if } \Theta_t > \sigma.
\end{align*}
Nevertheless, one can repeat the perturbation argument in \cite[Section 5]{Cap13} to once again deduce the result from the earlier case of distinct $(\theta_s)_{s=1}^r$. We conclude that
\begin{align*}
    \lim_{N \to \infty} \snorm{P_{\ker(\Theta_t \mbf{I}_N - \mbf{A}_N)}(\mbf{m}_N^{(k)})}_2^2 &\aseq 1-\frac{\sigma^2}{\Theta_t^2}, \qquad \forall k \in \left[\sum_{i = 1}^{t-1} m_i + 1, \sum_{i = 1}^{t} m_i\right] \text{ if } \Theta_t < -\sigma; \\
    \lim_{N \to \infty} \snorm{P_{\ker(\Theta_t \mbf{I}_N - \mbf{A}_N)}(\mbf{m}_N^{(N+1-k)})}_2^2 &\aseq 1-\frac{\sigma^2}{\Theta_t^2}, \qquad \forall k \in \left[\sum_{i = t + 1}^q m_i + 1, \sum_{i = t}^q m_i\right] \text{ if } \Theta_t > \sigma,
\end{align*}
which establishes \ref{spiked_eigenvectors_I}.
\end{proof}

\bibliographystyle{amsalpha}
\bibliography{master_bib}

\end{document}